\documentclass[12pt,a4paper]{amsart}
\usepackage{amsmath,amsthm,amssymb}
\usepackage{cite}
\usepackage{enumerate}
\usepackage[all]{xy}
\usepackage{geometry}
\usepackage[utf8]{inputenc}
\usepackage{cleveref}

\newcommand{\Ad}{\mathrm{Ad}}

\newcommand{\Gl}{\mathrm{GL}}
\DeclareMathOperator{\Hom}{\mathrm{Hom}}

\newcommand{\R}{\mathbb{R}}

\newcommand{\bs}{\backslash}
\newcommand{\cg}{{}_c G}
\newcommand{\cov}{\nabla}

\newcommand{\ddt}[1]{\left.\frac{d}{dt}\right|_{#1}}

\newcommand{\ev}{\mathrm{ev}}

\newcommand{\fb}{\mathbf{F}}
\newcommand{\forus}[1]{}
\newcommand{\from}{\colon}
\newcommand{\trlift}{\tilde{\lie}}
\newcommand{\glie}{\mathfrak{g}}

\newcommand{\gnktimes}{{\times}_{G_n^k}}
\newcommand{\gtimes}{\times_G}
\newcommand{\hlie}{\mathfrak{h}}
\newcommand{\id}{\mathrm{id}}
\newcommand{\im}{\mathrm{Im}}
\newcommand{\inv}{\mathrm{inv}}

\newcommand{\lie}{\mathcal{L}}
\newcommand{\mfn}{\mathcal{M}f_n}
\newcommand{\mf}{\mathcal{M}f}
\DeclareMathOperator{\opwedge}{\wedge}

\newcommand{\pr}{\mathrm{pr}}
\newcommand{\rlie}{\mathsterling}

\newcommand{\shuffles}{\mathrm{Sh}}

\newcommand{\vb}{\mathcal{V}}

\newcommand{\vf}{\mathfrak{X}}

\newcommand{\vl}{\mathrm{vl}}
\newcommand{\vpr}{\mathrm{vpr}}

\newtheorem{theorem}{Theorem}[section]
\newtheorem{proposition}[theorem]{Proposition}
\newtheorem{proposition-definition}[theorem]{Proposition-Definition}

\newtheorem{corollary}[theorem]{Corollary}
\theoremstyle{remark}

\newtheorem{example}[theorem]{Example}
\newtheorem{remark}[theorem]{Remark}
\crefname{example}{Example}{Examples}

\theoremstyle{definition}

\newtheorem{definition}[theorem]{Definition}

\title{Darboux-Lie derivatives}
\author[A. De Nicola]{Antonio De Nicola}
 \address{Dipartimento di Matematica, Università degli Studi di Salerno, Via Giovanni Paolo II 132, 84084 Fisciano, Italy}
 \email{antondenicola@gmail.com}

\author[I. Yudin]{Ivan Yudin}
 \address{University of Coimbra, CMUC, Department of Mathematics, 3001-501 Coimbra, Portugal}
 \email{yudin@mat.uc.pt}

\subjclass[2010]{53A55, 53C10, 58A20}

\keywords{Lie derivative; Darboux derivative; covariant derivative; natural bundle; G-structures; associated bundle}

\thanks{This work was partially supported by the Centre for Mathematics of the
University of Coimbra - UIDB/00324/2020, funded by the Portuguese Government through FCT/MCTES.
ADN is a member of GNSAGA - Istituto Nazionale di Alta Matematica. We would like to thank the anonymous referee for the useful suggestions.}

\begin{document}
\begin{abstract}
We introduce the Darboux-Lie derivative along a vector field of fiber bundle maps from natural bundles to associated fiber bundles and study its properties.
\end{abstract}
\maketitle
\section{Introduction}
This article is an accompanying paper to the forthcoming series of articles by the same authors on
the theory of $G$-structures.
In the course of reformulating the basic notions pertaining to $G$-structures on manifolds
in terms  of gauge equivalence classes of soldering forms,
we stumbled upon the absence of a properly developed calculus of derivatives for
such forms and for gauge transformations.
A \emph{soldering form} $\beta$ in this context refers to an isomorphism of vector bundles
$\beta \from TM \to P\gtimes V$, where $P$ is a principal $G$-bundle and
$V$ is a $G$-module (cf. \cite{alekseevsky}) and a \emph{gauge transformation} is a
section of the bundle $P\gtimes \cg$, where $\cg$ denotes the group $G$ with the conjugation action on itself. 

In this article we introduce and study properties of  Darboux-Lie and covariant Darboux-Lie derivatives of fiber bundle morphisms from $F(M)$ to $P\gtimes N$,
where $M$ is a manifold of dimension $n$, $F$ is a natural bundle on
$n$-dimensional manifolds, $P$ is a principal $G$-bundle over $M$, and $N$ is a
manifold equipped with a smooth left $G$-action.
The case of soldering forms is covered by $F(M) = TM$ and $N=V$ and the case of  gauge
transformations corresponds to  $F(M) =M$ and $N=\cg$.
We choose the generality of an arbitrary natural bundle with a shot for future
generalizations to high-order $G$-structures.

For a fixed $G$-module $W$ and $G$-equivariant $1$-form, the Darboux-Lie
derivative $\rlie_{\widetilde{X}}^\omega \beta \from F(M) \to
P\gtimes W$ of the bundle map $\beta \from
F(M) \to P\gtimes N$ is taken along a $G$-equivariant vector field
$\widetilde{X}$ on $P$. The complete definition, given in~\Cref{sec:darboux}, is slightly too technical to be
fully stated in the introduction.

If $P$ is equipped with a $G$-principal connection,
then the horizontal lift $X^H$ of a vector field $X$ on $M$ is $G$-equivariant.
We define the covariant Darboux-Lie derivative of $\beta$ along $X$ as
$\rlie^\omega_{X^H} \beta$.

We show in the article that various classically known derivatives (covariant,
Lie, Darboux) can be considered as special cases of (covariant) Darboux-Lie
derivatives.
We also prove various properties, including a characterization of (covariant)
Darboux-Lie derivatives in terms of flows and a suitable Leibniz
rule. We conclude with the proof of the Cartan magic formula for covariant
Darboux-Lie derivatives.

The paper is organized as follows. Section~\ref{sec:prel} contains
some preliminary material. In particular, we fix the notation for principal and
associated bundles, and recall the definition and main properties of natural
bundles.

In~\Cref{sec:alphaderivatives}, we discuss a metamathematical concept of
derivative. This discussion culminates in the definition of an
$\alpha$-derivative. We show that this notion incorporates Lie derivatives
of Janyška-Kolář and of Godina-Matteucci.
The (covariant) Darboux-Lie derivatives are defined in~\Cref{sec:darboux} as a
special case of $\alpha$-derivative.

Section~\ref{sec:examples} demonstrates how classical Lie derivatives of tensor fields and standard covariant derivatives of sections of vector bundles can be seen as particular instances of the Darboux-Lie derivative.

In Section~\ref{sec:vertical}, we calculate the Darboux-Lie derivative along
vertical $G$-invariant vector fields on $P$. This permits to compare
covariant Darboux-Lie derivatives calculated with respect to two $G$-principal
connections on $P$.

\Cref{sec:leibniz}  is dedicated to establishing various Leibniz rules for
the (covariant) Darboux-Lie derivative with respect to operations like fiber products, tensor products, and compositions of fiber bundle maps.

In \Cref{sec:magic}, we  derive a Cartan-type magic formula for covariant
Darboux-Lie derivative when applied to differential forms with values in an associated bundle.

In~\Cref{sec:future}, we announce some results related to $G$-structures.

\section{Preliminaries}\label{sec:prel}
\subsection{Principal bundles}

Let $G$ be a Lie group, $M$ a smooth manifold and $P$ a
principal $G$-bundle.
Then there is a unique
smooth map
\begin{equation*}
\begin{aligned}
\backslash \from P \times_M P & \to G \\
( p, p') & \mapsto p \bs p',
\end{aligned}
\end{equation*}
such that $p(p \backslash p') = p'$ for every pair $p$, $p'$ in the same fiber
of $P$. This map will be handy to explicitly write various maps involving principal
bundles.

Clearly, $p\backslash p = e_G$.
Moreover, for all  $g \in G$ and $(p,p') \in P\times_M P$, we have
\begin{equation}
\label{divisionoperatorproperties}
\begin{aligned}
(pg)\bs p' = g^{-1} ( p\bs p'),\quad p \bs (p'g) = (p\bs p') g,\quad p'\bs p =
(p\bs p')^{-1}.
\end{aligned}
\end{equation}

\subsection{Associated bundles}
Let $N$ be a manifold equipped with a smooth left $G$-action.
The \emph{associated fiber
bundle} $P \gtimes N$ over $M$  is defined as the quotient $(P\times N)/G$, where
the group
$G$ acts on $P \times N$ by $(p,y) g= (pg, g^{-1}  y)$ for $(p,y) \in
P\times N$ and $g\in G$.
We write $[p,y]$ for the $G$-orbit of $(p,y)$ in $P\times N$.
The projection from $P\gtimes N$ to $M$ is defined by sending the orbit $[p,y] =
\left\{\, (pg,g^{-1} y)  \,\middle|\,  g \in G \right\}$ to the common base
point of the elements $pg$, $g\in G$.

One can verify that the following map
is both well-defined and smooth:
\begin{equation*}
\begin{aligned}
\bs \from P\times_M (P\gtimes N) & \to N \\
(p, [p',y]) & \mapsto p\bs [p',y] :=  (p\backslash p')y
\end{aligned}
\end{equation*}
Notice that
$p\bs [p,y] = y$.
This map is particularly useful for explicit construction of maps between
associated bundles.
\subsection{Vector fields and their flows}
Given a vector field $X$ on $M$, we write $\gamma^X_x$ for the maximal integral
curve of $X$ that passes through $x \in M$ at time $0$. We denote its interval
of definition by $J^X_x \subset \mathbb{R}$. The union
\begin{equation*}
\begin{aligned}
D(X) := \bigcup_{x \in M} J^X_x \times \{x\} \subset \mathbb{R} \times M
\end{aligned}
\end{equation*}
is  the \emph{domain of the flow of $X$} and
the \emph{flow of $X$} is the map
\begin{equation*}
\begin{aligned}
\varphi_X \colon D(X) &\to M \\ (t, x) &\mapsto \gamma^X_x(t).
\end{aligned}
\end{equation*}
For each $t \in \mathbb{R}$, we define the open set
\begin{equation*}
\begin{aligned}
U^X_t := \{ x \in M \mid (t,x) \in D(X) \},
\end{aligned}
\end{equation*}
consisting of points whose integral curves are defined at time $t$. For a fixed
$t\in \R$, the \emph{flow operator} is the map
\begin{equation*}
\begin{aligned}
\Phi^t_X \colon U^X_t \to U^X_{-t}, \quad x \mapsto \varphi_X(t,x),
\end{aligned}
\end{equation*}
which is a diffeomorphism. These operators satisfy the usual composition rule:
\begin{equation*}
\begin{aligned}
\Phi^t_X \circ \Phi^s_X = \Phi^{t+s}_X \quad \text{on } U^X_s \cap U^X_{s+t}.
\end{aligned}
\end{equation*}

\subsection{Natural bundles}
A systematic treatment of natural bundles can be found in~\cite[Ch.~IX]{kolar}.
Here we will just recall the definition of natural bundles and their properties
that will be needed later.
It is important to mention that tensor bundles are instances of natural
bundles.

Write $\mf$ for the category of smooth manifolds with smooth maps between them
and $\mfn$   for its subcategory of manifolds  of dimension $n$
with local diffeomorphisms between them as morphisms. Denote by
$\mathcal{E}_n$ the embedding
of $\mfn$ into $\mf$.  One of the possible equivalent definitions of a natural
bundle is to say that a \emph{natural bundle}
is a pair $(F,p)$, where $F\from \mfn \to \mf$ is a functor
and $p \from  F \to \mathcal{E}_n$ is a natural transformation such that
\begin{enumerate}[(a)]
\item for each $M \in \mfn$ the map $p_{M} \from F(M) \to M$ turns $F(M)$
into a fiber bundle over $M$;
\item for each $f \from M \to N$ in $\mfn$ the map
$F(f) \from F(M) \to F(N)$ is a fiberwise diffeomorphism that covers
$f \from M \to N$, i.e.\ the commutative diagram
\begin{equation}
\label{cd:naturalbundle}
\begin{aligned}
\begin{gathered}
\xymatrix{
F(M) \ar[r]^{F(f)} \ar[d]_{p_M}& F(N) \ar[d]^{p_N} \\
M \ar[r]^{f} & N
}
\end{gathered}
\end{aligned}
\end{equation}
is a pull-back square\footnote{Natural transformations with this property are
known as \emph{Cartesian natural transformations} in Category Theory.}.
\end{enumerate}
A \emph{natural map}
from a  natural bundle $(F,p)$ and to a natural bundle $(F',p')$
is a natural transformation of functors $\eta \from F \to F'$ such that
$p = p' \circ \eta$.
It is customary to refer to a natural bundle just by its functor part $F$.

One defines a \emph{natural vector bundle} as a natural bundle $(F,p)$ such that each
$p_M \from F(M) \to M$ is a vector bundle and for each  $f \from M \to N$ in
$\mfn$ the map $F(f) \from F(M) \to F(N)$ is a morphism of vector bundles.
Similarly, one defines \emph{natural principal bundles}.

One of the key features of natural bundles is the existence of canonical lifts
of vector fields. However, a clear statement and proof of the existence of
canonical lifts for non-complete vector fields are difficult to find in the
literature. For the sake of rigor, we include the following proposition.
\begin{proposition}
\label{propcanonicallift}
Let $(F,p)$ be a natural bundle on $\mfn$.
For every $M \in \mfn$ there is a unique well-defined map
\begin{equation*}
\begin{aligned}
\mathcal{F} \from \vf(M) & \to \vf(F(M)) \\
X & \mapsto \mathcal{F}(X)
\end{aligned}
\end{equation*}
such that
for each $X \in \vf(M)$
\begin{enumerate}[(i)]
\item the vector field $\mathcal{F}(X)$ is a lift of
$X$ (called \emph{the canonical lift of $X$});
\item for all $x \in M$ and $y \in F(M)_x$ the intervals $J^X_x$ and
$J^{\mathcal{F}(X)}_y$
coincide;
\item for each $t \in \R$, the image of the embedding $F(U^X_t \hookrightarrow
M) \from F(U^X_t) \to F(M)$ coincides
with $U_t^{\mathcal{F}(X)}$ and
the diagram
\begin{equation}
\label{cdcharcterizingcanonicallift}
\begin{aligned}
\begin{gathered}
\xymatrix@C4em{
F(U_t^X) \ar[r]^{ F(\Phi_X^t) } \ar[d]_{F(U^X_t \hookrightarrow M)} &
F(U_{-t}^X) \ar[d]^{F(U^X_{-t} \hookrightarrow M)} \\
U^{\mathcal{F}(X)}_t \ar[r]^{\Phi_{\mathcal{F}(X)}^t} &
U^{\mathcal{F}(X)}_{-t}.
}
\end{gathered}
\end{aligned}
\end{equation}
commutes.
\end{enumerate}

\end{proposition}
\begin{proof}
Let $M \in \mfn$ and $X \in \vf(M)$.
The tricky point is that the family of diffeomorphisms $F(\Phi_X^t)$ is not a
$1$-parameter group on $F(M)$. Indeed, for a fixed $t$, the map $F(\Phi_X^t)$ is a
diffeomorphism between $F(U_t^X)$ and $F(U_{-t}^X)$, but $F(U_t^X)$ and
$F(U_{-t}^X)$ are not open subsets of $F(M)$ for a general natural bundle
$F$. To overcome this peculiarity we use the fact that every
natural bundle is isomorphic to an associated bundle of a suitable
jet bundle.

For $M \in \mfn$ and $k\ge 1$, write $\inv{J^k_0(\R^n,M)}$ for the space of invertible $k$-jets at
$0$ of smooth maps from $\R^n$ to $M$.
Denote by $G_n^k$ the subspace of $\inv J^k_0(\R^n,\R^n)$ consisting of jets that have
value $0$ at $0$. Then $G_n^k$ is a Lie group under composition of jets and $\inv J^k_0(\R^n,M)$ is a
principal $G_n^k$-bundle.

Given a smooth map $\psi \from \R^n \to M$ write $j^k(\psi)$ for the $k$-th jet of
$\psi$ at $0$.
Given a local diffeomorphism $f \from M \to N$,
we define
\begin{equation*}
\begin{aligned}
\inv J^k_0(\R^n,f) \from \inv J^k_0(\R^n,M) & \to \inv J^k_0(\R^n,N) \\
j^k_0(\psi) &\mapsto   j^k(f \circ \psi) .
\end{aligned}
\end{equation*}
This turns the functor $\inv J^k_0(\R^n,-)$ into a natural principal $G_n^k$-bundle on $\mfn$.

One can check that for every manifold $N$ equipped with a smooth $G_n^k$-action
the functor $\inv J^k_0(\R^n,-)\gnktimes N$ is a natural bundle.
It was proved
in~\cite{epstein} that every natural bundle is isomorphic to a natural bundle of
the form $\inv J^k_0(\R^n,-)\gnktimes N$.

Now, let
$k$ be a natural number  and $N$ a manifold
 equipped with
a smooth left $G_n^k$-action such that $F$ is isomorphic to $H= \inv J^k_0(\R^n,-)
\gnktimes N$.
We will show that there is $\mathcal{H} \from \vf(M) \to \vf(H(M))$ that
satisfies the properties stated in the proposition with $F$ and
$\mathcal{F}$ replaced by
$H$ and $\mathcal{H}$, respectively.
Then $\mathcal{F}(X)$ is obtained from $\mathcal{H}(X)$ via a natural isomorphism between
$F(M)$ and $H(M)$. The stated properties for $\mathcal{F}$ follow from the
corresponding properties for $\mathcal{H}$, since the isomorphisms are natural.

To get a vector field $\mathcal{H}(X)$ on $H(M)$, we construct its
flow on $H(M)$.
Write $\sigma_t$ for the embedding of $U^X_t$ into $M$.
Define the open subsets $\widetilde{U}_t$ of $H(M)$ by
\begin{equation*}
\begin{aligned}
\widetilde{U}_t = \left\{\, [j^k(\psi),y]\in H(M) \,\middle|\, \im(\psi) \subset U^X_t
\right\}
\end{aligned}
\end{equation*}
and the diffeomorphisms $\widetilde{\Phi}^t \from \widetilde{U}_t \to
\widetilde{U}_{-t}$ by
\begin{equation*}
\begin{aligned}
\widetilde{\Phi}^t ([j^k(\psi), y]) = [ j^k (\Phi_X^t \circ \psi), y].
\end{aligned}
\end{equation*}
It is easy to check that $\widetilde{\Phi}^t$ is a local flow.
We define $\mathcal{H}(X)$ as the unique vector field on $H(M)$ such that
$\Phi_{\mathcal{H}(X)}^t (z) = \widetilde{\Phi}^t(z)$ for all $z \in
\widetilde{U}_t$.
Write $p$ for the canonical projection from $H(M)$ to $M$.
Thus $p([j^k(\psi),y]) = \psi(0)$ for all $[j^k(\psi), y] \in H(M)$.
Then for all $x \in M$ and $[j^k(\psi), y]\in H(M)_x \cap \widetilde{U}_t $, we get
\begin{equation*}
\begin{aligned}
p( \Phi_{\mathcal{H}(X)}^t ([j^k(\psi), y])) =
p( [j^k(\Phi_X^t \circ \psi), y])
= (\Phi_X^t \circ \psi) (0) = \Phi_X^t (x).
\end{aligned}
\end{equation*}
Thus $\mathcal{H}(X)$ is a lift of $X$.
Moreover, $\widetilde{U}_t \subset U_t^{\mathcal{H}(X)}$, which implies that
$\bigcup_{t\in \R} \{t\} \times \widetilde{U}_t \subset D(\mathcal{H}(X))$.
Thus, for every $x \in M$ and $z \in H(M)_x$ the interval
$J^{\mathcal{H}(X)}_z$ contains the interval $J^X_x$. But, since
$\mathcal{H}(X)$ is a lift of $X$, also $J^{\mathcal{H}(X)}_z \subset J^X_x$.
Thus $J^{\mathcal{H}(X)}_z = J^X_x$ as claimed.
We also conclude that $U^{\mathcal{H}(X)}_t = \widetilde{U}_t$ for all
$t\in \R$.

It is clear that if $\psi'\from \R^n \to U^X_t$ then $\sigma_t \circ \psi'
\from \R^n \to M$ has its image in $U^X_t$. Also for every $\psi \from \R^n \to
M$ with $\im(\psi) \subset U^X_t$ there is a unique $\psi' \from \R^n \to
U^X_t$ such that $\psi = \sigma_t \circ \psi'$. This shows that
$\widetilde{U}_t$, and hence also $U^{\mathcal{H}(X)}_t$,  is the image of $H(\sigma_t)$.
The commutativity of the diagram~\eqref{cdcharcterizingcanonicallift} follows
from a straightforward computation.
\end{proof}
Let $\eta \from F' \to F$ be a natural map between natural bundles on $\mfn$ and $M\in
\mfn$. Then for every $X \in
\vf(M)$, the vector fields $\mathcal{F}'(X)$ and $\mathcal{F}(X)$ are
$\eta_M$-related.  To see this, let $x \in M$
 and $y \in F'(M)_x$. We have
\begin{equation}
\label{tetamfprimexy}
\begin{aligned}
T\eta_M (\mathcal{F}'(X)_y) = \ddt{t=0} \eta_M \left(
\Phi^t_{\mathcal{F}'(X)}y \right) \\
\mathcal{F}(X)_{\eta_M(y)} = \ddt{t=0} \Phi^t_{\mathcal{F}(X)} (\eta_M(y)).
\end{aligned}
\end{equation}
For $t \in \R$, write $\sigma_t$ for the embedding $U_t^X \hookrightarrow
M$. Then $F'(\sigma_t)$ is a diffeomorphism between $F'(U_t^X)$ and
$U_t^{\mathcal{F}'(X)} \subset F'(M)$.
Denote by $y_t$ the unique element of $F'(U_t)$ such that $F'(\sigma_t)
(y_t) = y$. Using the commutativity of~\eqref{cdcharcterizingcanonicallift} and
the naturality of
$\eta_M$, we get
\begin{equation*}
\begin{aligned}
\eta_M(\Phi_{\mathcal{F}'(X)}^t y ) & =
\eta_M(\Phi_{\mathcal{F}'(X)}^t \circ F'(\sigma_t) (y_t) )  =
\eta_M(  F'(\sigma_{-t}) \circ F'(\Phi_{X}^t)(y_t) ) \\& =
  F(\sigma_{-t}) \circ F(\Phi_{X}^t) (\eta_{U_t}(y_t) )  =
 \Phi_{\mathcal{F}(X)}^t \circ F(\sigma_{t}) (\eta_{U_t}(y_t) ) \\& =
 \Phi_{\mathcal{F}(X)}^t (\eta_{M}(F'(\sigma_t)(y_t)) )  =
 \Phi_{\mathcal{F}(X)}^t (\eta_{M}(y)).
\end{aligned}
\end{equation*}
Thus~\eqref{tetamfprimexy} implies that
\begin{equation}
\label{liederivativeofnaturalmapiszero}
\begin{aligned}
T\eta_M (\mathcal{F}'(X)_y) = \mathcal{F}(X)_{\eta_M(y)},
\end{aligned}
\end{equation}
i.e.\ the vector fields $\mathcal{F}'(X)$ and $\mathcal{F}(X)$ are
$\eta_M$-related.
\section{Derivatives}
\label{sec:alphaderivatives}
We start with a discussion of what one might expect from a notion of
derivative.
 Given a map $h$ between two smooth manifolds, the tangent map  $Th$
captures
the first-order behavior of $h$. The drawback of $Th$  is
that it retains $h$ inside it. A metamathematical idea  of a derivative of
$h$ would be $Th$ with $h$ stripped out and only the first-order
information retained.
It is evident how to do this in the case of $\R$-valued smooth functions.
Namely, given a smooth function $f \from M \to \R$ the  derivative  $df\from TM
\to \R$ is determined by
\begin{equation*}
\begin{aligned}
 (T_x f ) (X) = df(X)\ddt{t=f(x)}.
\end{aligned}
\end{equation*}
In other words $df = dt \circ Tf$. The key point of the
definition is the existence of the canonical nowhere zero $1$-form $dt$ on $\R$.
More generally,
let $N$ be an $n$-dimensional manifold and $\omega\from TN \to \R^n$ a $1$-form
on $N$ such that $\omega_y \from T_y N  \to \R^n$ is an isomorphism for every
$y \in N$.
Then for a smooth map $h \from M \to N$, one can define its derivative as~$\omega\circ Th$.
Notice that a form $\omega$ with the above property exists if and only if
$N$ is a parallelizable manifold. For this reason we call such a form
\emph{a parallelization form}. It corresponds to the \emph{global coframe} $(\pr_1\circ\omega, \dots, \pr_n\circ\omega)$.

Of course, $\omega \circ Th$ depends on the choice of the parallelization
form $\omega$. For example, in the case $N$ is  a Lie group $H$, there are two natural
choices for $\omega$: either left-invariant  or right-invariant
Maurer-Cartan form.
In this case the resulting derivative is called \emph{left (right) Darboux
derivative}, respectively.

Another standard example of a derivative is the \emph{covariant derivative} of a
section $s\from M \to P\gtimes V$, where $P$ is a principal $G$-bundle with a
fixed
principal connection and
$V$ is a $G$-module.
In \cite{kolar_article}, Jany\v{s}ka and Kol\'{a}\v{r} developed a general theory of Lie
derivatives
that includes as special cases the usual Lie derivatives and covariant
derivatives.

Below, we introduce the notion of $\alpha$-derivatives that captures the idea of derivative
explained above.
For this, we use the following auxiliary construction.

Let $M_1$ and $M_2$ be smooth manifolds and $h\from M_1 \to M_2$ a smooth map.
For a pair of vector fields $X_1 \in \vf(M_1)$, $X_2 \in \vf(M_2)$, we call the map
\begin{equation*}
\begin{aligned}
\trlift_{(X_1,X_2)}h  \from M_1 &\to TM_2\\
x & \mapsto (T_x h)(X_1)-(X_2)_{h(x)}.
\end{aligned}
\end{equation*}
\emph{the Trautman lift of $h$}.
Alternatively (cf.~\cite[Lemma~47.2]{kolar}), we have
\begin{equation}
\label{genliederviaflows}
\begin{aligned}
(\trlift_{(X_1,X_2)}h)_x = \ddt{t=0} \Phi^{-t}_{X_2} \circ h \circ
\Phi^t_{X_1} (x).
\end{aligned}
\end{equation}
\begin{remark}
The map $\trlift_{(X_1,X_2)}h$ was initially introduced by Trautman
in~\cite{trautman} without a specific name. It was later referred to as the
``generalized Lie derivative''
 in a series of articles by Kol\'{a}\v{r} and his coauthors, with the main
results summarized in \cite[Ch.~XI]{kolar}. However, we have chosen to deviate
from this terminology because ``generalized Lie derivatives'' do not align with
our concept of a derivative and never specialize to Lie derivatives. Instead, they serve as an intermediate step in the construction of Lie derivatives.
\end{remark}
The Trautman lift is particularly useful when  $h: \fb_1 \to \fb_2$ is a map of
fiber bundles over a manifold $M$ and
 $X_1$, $X_2$ are lifts of the same
vector field on $M$.
In this case, the values of $\trlift_{(X_1,X_2)}h$ lie
in the vertical subbundle $\vb\fb_2$ of the tangent bundle $T\fb_2$.
\begin{definition}
Let $h \from \fb_1 \to \fb_2$ be a bundle morphism over $M$ and $(X_1, X_2)$
a pair of lifts of the same vector field on $M$ to $\fb_1$ and to $\fb_2$,
respectively.
For a vector bundle $E$ over $M$ and a vector bundle map $\alpha\from
\vb\fb_2 \to E$ that covers the projection $\fb_2 \to M$, we denote the
composite
\begin{equation*}
\begin{aligned}
\alpha \circ \trlift_{(X_1,X_2)} h \from \fb_1 \to E
\end{aligned}
\end{equation*}
by $\lie^{\alpha}_{(X_1,X_2)} h$ and call it the \emph{$\alpha$-derivative of
$h$ along $(X_1,X_2)$}.
\end{definition}

\begin{remark}
The $\alpha$-derivative carries all the first-order information about $h$ provided that, for each base point $x \in M$ and each $y \in \fb_{2,x}$, the map
\begin{equation}
\alpha_y : \vb_y \fb_2 \to E_x
\end{equation}
is injective.

A potential issue arises if the images of $\alpha_y$ vary with the choice of $y$
inside the same fiber $\fb_{2,x}$. In that case, the $\alpha$-derivative may
contain \emph{more} than just differential information: it might even allow one
to recover the actual values of $h$ in the fiber.

To avoid this, one can assume that, for each $x \in M$, the image of $\alpha_y$
is the same for all $y \in \fb_{2,x}$. Imposing injectivity and this
independence of the image ensures that $\operatorname{Im} \alpha$ forms a vector
subbundle of $E$.
Thus, the notion of $\alpha$-derivative fully aligns with the intended
philosophy of a derivative if and only if $\alpha$ induces a fiber-wise
isomorphism onto a subbundle of $E$. Of course, by replacing $E$ with this
subbundle, nothing is lost.

In all instances we are aware of where the $\alpha$-derivative specializes to a
previously studied derivative (e.g., Lie or covariant derivative), the
corresponding choice of $\alpha$ is a fiberwise isomorphism onto $E$. We have
nevertheless not imposed this assumption in the definition, since doing so would
prevent a clean formulation of the Leibniz rule for $\alpha$-derivatives.
\end{remark}

A family of suitable choices of $\alpha$ can be obtained using vertical
splittings.
A \emph{vertical splitting} for a fiber bundle $\fb_2$ over $M$ is an isomorphism
of vector bundles $\beta \from \vb \fb_2 \to \fb_2 \times_M E $ over $\fb_2$ for some vector bundle
$E$ over $M$.
Then $\alpha = \pr_2 \circ \beta$ is a vector bundle map from $\vb\fb_2$ to
$E$  that covers the
projection $\fb_2 \to M$ such that the maps $\alpha_y$ are isomorphisms.

If $\fb_2$ is a vector bundle on its own, then the inverse
of
the \emph{vertical lift}
\begin{equation*}
\begin{aligned}
\vl_{\fb_2} \from \fb_2 \times_M \fb_2 & \to \vb\fb_2 \\
(u,v) & \mapsto \ddt{t=0}(u + vt)
\end{aligned}
\end{equation*}
is an example of vertical splitting. The composite $\pr_2\circ
\vl_{\fb_2}^{-1}$ is called the \emph{vertical projection} and is denoted by
$\vpr_{\fb_2}$.
The $\vpr_{\fb_2}$-derivative $\lie^{\vpr_{\fb_2}}_{(X_1,X_2)} h$ is called the Lie
derivative of $h$ with respect to $X_1$ and $X_2$ in~\cite[Section~47.7]{kolar}.
This derivative was introduced by Jany\v{s}ka and Kol\'{a}\v{r}
in~\cite{kolar_article}.

Now, we return to the context of a general fiber bundle $\fb_2$, but assume that
$\fb_1 = M$ with the projection map $\id_M$.
In this setup $h \from M \to \fb_2$ is just a section of $\fb_2$.
If $X_2 \in \vf(\fb_2)$ is a projectable vector field lifting $X \in
\vf(M)$ and $\beta \from \vb\fb_2 \to \fb_2 \times_M E$ is a vertical
splitting, then the $(\pr_2\circ \beta)$-derivative of $h$ along $(X,X_2)$ was
called the (restricted) Lie derivative of $h$ with respect to $X_2$
by Godina and Matteucci in~\cite{godina1,godina2}.

The Lie derivatives of Janyška-Kolář and of Godina-Matteucci can be further
specialized to get classical Lie and covariant derivatives.
\begin{example}[\emph{Lie derivative}]
\label{liederivativeexample}
Suppose $F$ is a
natural bundle and $E$ a natural vector bundle on $\mfn$.
Given $M \in \mfn$, a vector field $X \in \vf(M)$ and a smooth map $h \from
F(M) \to E(M)$, \emph{the Lie derivative $\lie_X h$ of $h$ along $X$} is
the $\vpr_{E(M)}$-derivative of $h$ along $(\mathcal{F}(X),
\mathcal{E}(X))$. It coincides with the Lie derivative defined in textbooks in
the case both $F(M)$ and $E(M)$ are tensor bundles over $M$.
\end{example}
\begin{example}[\emph{Covariant derivative}]
\label{covariantderivativeexample}
If $E$ is a vector bundle equipped with a linear connection
then the covariant derivative $\cov_X s$ of a section $s \from M \to E$ is the
$\vpr_E$-derivative of $s$ along $(X,X^\cov)$, where $X^\cov$ is the horizontal lift
of $X \in \vf(M)$ (see \cite[Section 47.5]{kolar}).
\end{example}

Despite the generality of $\alpha$-derivatives,
it is still possible to obtain some results for them. Namely, we will show
that there is a version of Leibniz rule for an $\alpha$-derivative of a
composition.

We start with the Leibniz rule for the Trautman lift of a composition.
The following proposition follows directly from the definition of the Trautman lift.
\begin{proposition}
\label{liederivativecomposition}
Let $M_1$, $M_2$, $M_3$ be manifolds and $X_i \in \vf(M_i)$, $i=1$, $2$,
$3$.
For any smooth maps $h_1 \from M_1 \to M_2$, $h_2 \from M_2 \to M_3$, we have
\begin{equation}
\label{genliderLeibniz}
\begin{aligned}
\trlift_{(X_1,X_3)} (h_2 \circ h_1) = Th_2 \circ \trlift_{(X_1,X_2)} h_1 +
(\trlift_{(X_2,X_3)}h_2) \circ h_1.
\end{aligned}
\end{equation}
\end{proposition}
Notice that the vector field ${X}_2$ is not present on the left-hand
side of~\eqref{genliderLeibniz} and thus can be chosen arbitrarily on
the right-hand side.

The above proposition immediately implies the following result.
\begin{corollary}
\label{alphaLeibniz}
Let $\fb_1$, $\fb_2$ and $\fb_3$ be fiber bundles over a manifold
$M$ and $h_1\from \fb_1 \to \fb_2$, $h_2 \from \fb_2 \to \fb_3 $ fiber
bundle maps. Suppose $\alpha \from \vb\fb_3 \to E$ is a vector bundle map
covering the projection from $\fb_3$ to $M$. Write $h_2^* \alpha$ for the
composite of $\alpha$ with the restriction of $Th_2$ to $\vb\fb_2$. It is a
vector bundle map covering the projection from $\fb_2$ to $M$. For every triple
$(X_1,X_2,X_3) \in \vf(\fb_1)\times \vf(\fb_2) \times \vf(\fb_3)$ of
projectable  vector
fields lifting the same vector field on $M$, we get
\begin{equation}
\label{eq:alphaLeibniz}
\begin{aligned}
\lie^{\alpha}_{(X_1,X_3)} (h_2 \circ h_1) = \lie^{h_2^*\alpha}_{(X_1,X_2)} h_1
+ (\lie^{\alpha}_{(X_2,X_3)}h_2) \circ h_1.
\end{aligned}
\end{equation}
\end{corollary}

\section{ Darboux-Lie derivatives}\label{sec:darboux}
The peculiarity of $\alpha$-derivative is that it is defined along pairs of vector
fields on fiber bundles. This is a striking contrast with classical derivatives
that are defined along vector fields on the base manifold.
The utility of $\alpha$-derivative, as demonstrated
by Examples~\ref{liederivativeexample} and \ref{covariantderivativeexample},  lies in the fact
that usual derivatives can
be expressed in terms of $\alpha$-derivative. Thus $\alpha$-derivatives provide a
framework to introduce new types of derivatives. By
specializing the notion of $\alpha$-derivative, we introduce
the notion of Darboux-Lie derivative taken
along invariant vector fields on a principal bundle over the base manifold.
We accompany
this notion by the notion of covariant Darboux-Lie
derivative taken along vector fields on the base
manifold, but additionally depending on the choice of a principal connection.


Let
$G$ be a  Lie group, $P$ a  principal $G$-bundle  and $N$  a manifold
equipped with a smooth left $G$-action.
We will identify $\vb (P\gtimes N)$ with $P\gtimes TN$ via
the isomorphism
\begin{equation}
\label{nuiso}
\begin{aligned}
\nu \from \vb(P\gtimes N) &\to P\gtimes TN\\
\ddt{t=0}[\gamma_P(t),\gamma_N(t)]&\mapsto
\left[\gamma_P(0),\ddt{t=0}(\gamma_P(0)\bs\gamma_P(t))\gamma_N(t)\right],
\end{aligned}
\end{equation}
where $\gamma_P$ is a vertical curve in $P$ and $\gamma_N$ is an arbitrary
curve in $N$, both defined on the same open interval containing $0$.

Suppose $\widetilde{X}$ is a $G$-invariant vector field on the
principal $G$-bundle  $P$.
Write $\widetilde{X}_N$ for the vector field on $P\gtimes N$, whose
flow is given by
$\Phi_{\widetilde{X}_N}^t [ p, y] = [ \Phi_{\widetilde{X}}^t p , y]$.
Notice that
the vector fields $\widetilde{X}$  and $\widetilde{X}_N$ are both projectable and
correspond to the same vector field $X$ on $M$.

Let $F$ be a natural bundle on $\mfn$.
Recall that $\mathcal{F}(X)$ is the canonical lift of
$X$ to $F(M)$, whose existence is proved in Proposition~\ref{propcanonicallift}.
For a fiber bundle map $h\from F(M) \to P\gtimes N$,
we denote the Trautman lift $\trlift_{(\mathcal{F}(X),
\widetilde{X}_N)}h$  by  $\tilde\rlie_{\widetilde{X}} h $.

Let $V$ be a $G$-module. For a $G$-equivariant $1$-form $\omega \from TN \to V$
define
\begin{equation*}
\begin{aligned}
\id \gtimes \omega \from P \gtimes TN & \to P \gtimes V \\
 [ p, Y] & \mapsto [p, \omega(Y)].
\end{aligned}
\end{equation*}
It is clear that $\id\gtimes \omega$ is a vector bundle map covering the
projection $P\gtimes N \to M$.
\begin{definition}
Let $F$ be a natural bundle on $\mfn$, $M$ an $n$-dimensional smooth manifold, $P$ a principal
$G$-bundle, and $N$ a smooth manifold equipped with a $G$-action.
Suppose $\omega \from TN \to V$ is a $G$-equivariant $1$-form and
$\widetilde{X} \in \vf(P)^G$. For a bundle map
$h\from F(M) \to P\gtimes N$, the \emph{Darboux-Lie derivative}
\begin{equation*}
\begin{aligned}
\rlie_{\widetilde{X}}^{\omega} h \from F(M) \to P\gtimes V
\end{aligned}
\end{equation*}
\emph{of $h$ along $\widetilde{X}$ with respect to $\omega$} is the
$\alpha$-derivative $\lie_{(\mathcal{F}(X),
\widetilde{X}_N)}^{\alpha}  h$ with $\alpha = (\id \gtimes \omega) \circ \nu$,
i.e.\ $\rlie_{\widetilde{X}}^{\omega} h = (\id \gtimes \omega) \circ \nu \circ
\tilde\rlie_{\widetilde{X}} h$.
\end{definition}
\begin{definition}
If $P$ is additionally equipped with a $G$-principal connection, we define
the \emph{covariant Darboux-Lie derivative} $\lie^{\omega,\cov}_X h$ of
$h$ along $X \in \vf(M)$ with respect to $\omega$ to be $\rlie_{X^H}^\omega h$,
where $X^H$ is the horizontal lift of $X$ to $P$.
\end{definition}

If $N$ is a $G$-module $V$,  then there is the canonical $G$-equivariant  parallelization form
$\vpr_V := \pr_2 \circ \vl_V^{-1}$ on
$V$, where
\begin{equation*}
\begin{aligned}
\vl_V \from V \oplus V & \to TV \\
(v,w) & \mapsto \ddt{t=0}( v + wt).
\end{aligned}
\end{equation*}
\begin{remark}
\label{calculusderivative}
If $\gamma$ is a curve in $V$, then $\vpr_V(\dot\gamma(0))$ coincides with the
derivative of $\gamma(t)$ at $0$ as it is usually defined in standard Calculus
courses.
\end{remark}
Another interesting case is when $N$  coincides with  $G$ considered with
the conjugation action.
In this case, we can take $\omega$ to be the left-invariant Maurer-Cartan form
on $G$, defined by
\begin{equation*}
\begin{aligned}
\omega_{MC} \from TG & \to \glie = T_e G\\
T_g G \ni Z & \mapsto (T_e L_g)^{-1} (Z).
\end{aligned}
\end{equation*}
In these two cases we will skip  $\omega$ in $\rlie_{\widetilde{X}}^\omega$ and
write $\rlie_{\widetilde{X}}h$ instead of $\rlie^{\vpr_V}_{\widetilde{X}}h$ or
$\rlie^{\omega_{MC}}_{\widetilde{X}}h$.
Similarly, we write $\lie_X^{\cov} h$ for $\lie_X^{\vpr_V,\cov}h$ and for
$\lie_X^{\omega_{MC},\cov}h $.

The Darboux-Lie derivatives
$\rlie^{\omega}_{\widetilde{X}}$ can be expressed in terms of flows.
\begin{proposition}
\label{proprlieomegaflow}
Let $h \from F(M) \to P\gtimes N$ be a fiber bundle map over $M$.
For all $x \in M$, $p \in P_x$ and $y\in F(M)_x$
\begin{equation}
\label{rlieomegaviaflows}
\begin{aligned}
\rlie_{\widetilde{X}}^\omega h(y) = \left[   p, \omega\left(\ddt0 (\Phi_{\widetilde{X}}^t p) \bs h (
\Phi_{\mathcal{F}(X)}^t(y)) \right)
\right] \in P\gtimes V.
\end{aligned}
\end{equation}
\end{proposition}
\begin{proof}
From~\eqref{genliederviaflows}, we get
\begin{equation*}
\begin{aligned}
\tilde\rlie_{\widetilde{X}} h (y) & =
 \ddt0 \Phi_{\widetilde{X}_N}^{-t} \circ h \circ \Phi_{\mathcal{F}(X)}^t (y)
\\& =
 \ddt0 (\Phi_{\widetilde{X}}^{-t} \gtimes \id) [  \Phi_{\widetilde{X}}^t p, (\Phi_{\widetilde{X}}^t p) \bs h ( \Phi_{\mathcal{F}(X)}^t (y))] \\& =
 \ddt0  [   p, (\Phi_{\widetilde{X}}^t p) \bs h ( \Phi_{\mathcal{F}(X)}^t
(y))].
\end{aligned}
\end{equation*}
Now the result follows from
$\rlie^\omega_{\widetilde{X}}h = (\id \gtimes \omega) \circ \nu \circ
\tilde\rlie_{\widetilde{X}}h$.
\end{proof}
\section{Examples and relation to other definitions}\label{sec:examples}
Many known derivatives can be expressed in terms of
Darboux-Lie derivative.
\subsection{Lie derivative}
First we consider the case of Lie derivatives as defined
in Example~\ref{liederivativeexample}.
Let $F$ be a natural bundle and $E$ a natural vector bundle.
Then $E$ is isomorphic to, and thus can be replaced by, the functor $\inv
J^k_0(\R^n,-) \gnktimes V$ for a suitable $G^k_n$-module $V$.
For $X \in \vf(M)$, denote by $X^c$ the canonical lift of $X$ to $\inv
J^k_0(\R^n,M)$. Then $X^c$ is $G^k_n$-invariant, and $X^c_V$
coincides with the canonical lift $\mathcal{E}(X)$ of $X$ to $E(M)$.
Hence for $h\from F(M) \to E(M) = \inv J^k_0(\R^n,M) \gnktimes V$
\begin{equation}
\label{lieXheqrlieEXh}
\begin{aligned}
\lie_X h = \lie^{\vpr_{E(M)}}_{(\mathcal{F}(X),\mathcal{E}(X))} h=
\rlie^{\vpr_V}_{X^c}h.
\end{aligned}
\end{equation}
\subsection{Covariant derivative}
Let $M$ be a manifold and $E$ an $m$-dimensional vector bundle on $M$. For an
$m$-dimensional vector space $W$, we denote by $L(W,E)$ the bundle of
$W$-frames in $E$, thus $L(W,E)$ is the set of maps $f \from W \to E$ that
induce an isomorphism  between $W$ and $E_x$ for some $x \in
M$ depending on $f$.
The set $L(W,E)$ is a principal $\Gl(W)$-bundle over $M$ and $E$ is canonically isomorphic
to the associated bundle $L(W,E) \times_{\Gl(W)} W$ via
the evaluation map
\begin{equation*}
\begin{aligned}
\ev_E \from L(W,E) \times_{\Gl(W)} W & \to E \\
[f,w] &\mapsto f(w).
\end{aligned}
\end{equation*}
Let $\cov$ be a linear connection on $E$. For a vector field $X$ on M, denote
its
horizontal lift to $E$ by $X^\cov$. We get an induced connection on the
principal $\Gl(W)$-bundle $L(W,E)$, for which the horizontal lift $X^H$ of
$X$ has the
flow
\begin{equation*}
\begin{aligned}
\Phi_{X^H}^t \from L(W,E) & \to L(W,E) \\
f & \mapsto ( w \mapsto \Phi_{X^\cov}^t (f(w))).
\end{aligned}
\end{equation*}
We can express the covariant derivatives
from Example~\ref{covariantderivativeexample} in terms of
covariant-Darboux-Lie derivative.
\begin{proposition}
Let  $\cov$ be a linear
connection on $E$. If $h \from M \to E$ is a section of a vector bundle $E$,
then for every $X \in
\vf(M)$
\begin{equation}
\label{covariantasDarbouxLie}
\begin{aligned}
\cov_X  h= \ev_E \circ \lie_{X}^\cov (\ev_E^{-1} \circ h).
\end{aligned}
\end{equation}
\end{proposition}
\begin{proof}
As we already explained in Example~\ref{covariantderivativeexample}, we have
\begin{equation*}
\begin{aligned}
 \cov_X h = \lie^{\vpr_E}_{(X,X^\cov)} h = \vpr_E \circ \trlift_{(X,X^\cov)}h
 = \pr_2 \circ \vl_E^{-1} \circ \trlift_{(X,X^\cov)}h.
\end{aligned}
\end{equation*}
Hence, for every $x \in M$,
we get by~\eqref{genliederviaflows}
\begin{equation}
\label{covXhx}
\begin{aligned}
 \cov_X h  (x) & =  \vpr_E \left(
\ddt{t=0} \Phi_{X^\cov}^{-t} \circ h \circ \Phi_X^t(x))
\right)
\\& =  \vpr_E \left(
\ddt{t=0} \Phi_{X^\cov}^{-t} \circ h \circ \Phi_X^t(x)) \right).
\end{aligned}
\end{equation}
On the other hand
\begin{equation}
\label{lieXcovevEh}
\begin{aligned}
 \lie_{X}^\cov (\ev_E^{-1} \circ h) = (\id \times \vpr_W) \circ \nu \circ
\trlift_{(X,X^H_W)} (\ev_E^{-1} \circ h),
\end{aligned}
\end{equation}
where $\nu \from \vb (L(W,E) \gtimes W) \xrightarrow{\cong} L(W,E) \gtimes TW$
is given by~\eqref{nuiso}.
By~\eqref{genliederviaflows}, we get
\begin{equation}
\label{trliftXXHWevEhx}
\begin{aligned}
\trlift_{(X,X^H_W)} (\ev_E^{-1} \circ h) (x) = \ddt{t=0} \left(\Phi_{X^H_W}^{-t}
\circ \ev_E^{-1} \circ h \circ \Phi_X^t\right)(x).
\end{aligned}
\end{equation}
Write $c(t)$ for $\Phi_X^t(x)$.
Observe that $h( \Phi_X^t(x)) \in E_{c(t)}$ and $\Phi_{X^H}^t f \from W
\xrightarrow{\cong} E_{c(t)} $ is a frame at $c(t)$. Thence
\begin{equation*}
\begin{aligned}
\ev_{E}^{-1} \left( h \left( \Phi_X^t(x) \right) \right) = \left[
\Phi_{X^H}^t f, (\Phi_{X^H}^t f)^{-1} \circ h \circ \Phi_X^t(x)
\right].
\end{aligned}
\end{equation*}
Plug in this into~\eqref{trliftXXHWevEhx} and using the definition of $X^H_W$,
we get
\begin{equation*}
\begin{aligned}
\trlift_{(X,X^H_W)} (\ev_E^{-1} \circ h) (x) = \ddt{t=0}
\left[ f, (\Phi_{X^H}^t f)^{-1} \circ h \circ \Phi_X^t(x)
\right].
\end{aligned}
\end{equation*}
Applying~\eqref{nuiso} and~\eqref{lieXcovevEh}, we
get
\begin{equation*}
\begin{aligned}
 \lie_{X}^\cov (\ev_E^{-1} \circ h)(x) =
\left[ f, \vpr_W\left(\ddt{t=0} (\Phi_{X^H}^t f)^{-1} \circ h \circ \Phi_X^t(x)
\right) \right].
\end{aligned}
\end{equation*}
Therefore
\begin{equation}
\label{evElieXcovevEhx}
\begin{aligned}
\ev_E \left(\lie_{X}^\cov (\ev_E^{-1} \circ h) (x) \right)=
 f\left( \vpr_W\left(\ddt{t=0} (\Phi_{X^H}^t f)^{-1} \circ h \circ \Phi_X^t(x)
\right) \right).
\end{aligned}
\end{equation}
Write $\tilde{c}(t)$ for $h(\Phi_X^t(x))$. We have $\tilde{c}(t) \in
E_{c(t)}$. Since $\Phi_{X^H}^t f (w) = \Phi_{X^\cov}^t(f( w))$ for all $w \in W$,
we get
\begin{equation*}
\begin{aligned}
(\Phi_{X^H}^{t}f )^{-1} (\tilde{c}(t) ) & =
f^{-1} \left( \Phi_{X^\cov}^{-t} (\tilde{c}(t)) \right).
\end{aligned}
\end{equation*}
Therefore
\begin{equation}
\label{evEsecond}
\begin{aligned}
\ev_E \left(\lie_{X}^\cov (\ev_E^{-1} \circ h) (x) \right)=
 f\left( \vpr_W\left(\ddt{t=0} f^{-1}\left(\Phi_{X^\cov}^{-t} \circ h \circ \Phi_X^t(x)
\right) \right)\right).
\end{aligned}
\end{equation}
It is left to check that the right hand sides of~\eqref{covXhx}
and~\eqref{evEsecond} coincide, i.e.\ that
\begin{equation*}
\begin{aligned}
 f\left( \vpr_W\left(\ddt{t=0} f^{-1}\left(\Phi_{X^\cov}^{-t}\circ h \circ \Phi_X^t(x)
\right) \right)\right) =
 \vpr_E \left(
\ddt{t=0} \Phi_{X^\cov}^{-t} \circ h \circ \Phi_X^t(x)) \right).
\end{aligned}
\end{equation*}
Let $\gamma(t)$ be the curve  $f^{-1}(\Phi_{X^\cov}^{-t} \circ h \circ
\Phi_X^t(x))) $ in $W$.
Then the last equation can be written as
\begin{equation*}
\begin{aligned}
f \left( \vpr_W (\dot{\gamma}(0) )\right)  = \vpr_E \left( \ddt{t=0} (f\circ
\gamma)(t) \right)
\end{aligned}
\end{equation*}
or, equivalently, as
\begin{equation*}
\begin{aligned}
f \left( \vpr_W (\dot{\gamma}(0) )\right)   = \vpr_E \left( Tf \left( \dot{\gamma}(0) \right) \right).
\end{aligned}
\end{equation*}
Hence, we only have to ensure that $f \circ \vpr_W = \vpr_E \circ Tf$.
Every element of $TW$ can be written as $\ddt{t=0} (w_0 + w_1t)$ for suitable
$w_0$, $w_1 \in W$.
We get
\begin{equation*}
\begin{aligned}
f \circ \vpr_W \left( \ddt{t=0}(w_0 + w_1t) \right) & = f \circ \pr_2 \circ
\vl_W^{-1} \left( \ddt{t=0}(w_0 + w_1t) \right)
\\ & = f (\pr_2 (w_0,w_1)) = f(w_1)
\end{aligned}
\end{equation*}
and, using that $f$ is linear,
\begin{equation*}
\begin{aligned}
\vpr_E \left( Tf \left( \ddt{t=0} (w_0 + w_1 t) \right) \right)
& = \pr_2 \circ \vl_E^{-1} \left( \ddt{t=0} (f(w_0)+ f(w_1)t \right)
\\& = \pr_2 \left( f(w_0), f(w_1) \right) = f(w_1).
\end{aligned}
\end{equation*}
Hence the result.
\end{proof}

\section{Darboux-Lie derivatives with respect to vertical right-invariant vector fields}\label{sec:vertical}
Write $\glie$ for the Lie algebra of the group $G$.
For each section $a$ of $P\gtimes \glie$, we denote by $X^a$
the vertical vector field
on $P$ whose flow is given by
\begin{equation*}
\begin{aligned}
\Phi_{X^a}^t p = p \exp ( (p\bs a(x)) t),\quad \mbox{for } x \in M, p \in P_x.
\end{aligned}
\end{equation*}
It is a $G$-invariant vector field on $P$. Conversely, one can show that if
$\widetilde{X}$ is a
$G$-invariant vertical vector field  on $P$ then it coincides with $X^a$, where
\begin{equation*}
\begin{aligned}
a \from M & \to P\gtimes \glie \\
x & \mapsto \left[ p, \ddt0 p \bs \Phi^t_{\widetilde{X}}p \right],
\end{aligned}
\end{equation*}
for an arbitrary $p \in P_x$.

Let $V$ be a $G$-module. The induced action of $\dot\gamma(0) \in \glie$ on $v \in V$ is
defined by
\begin{equation}
\label{liealgebraactiononV}
\begin{aligned}
 \dot\gamma(0) v =\vpr_V \left( \ddt0 \gamma(t)v\right).
\end{aligned}
\end{equation}
For $a \in \Gamma(M, P\gtimes \glie)$ and $h \from F(M) \to P\gtimes V$, we
define $a \cdot h \from F(M) \to P\gtimes V$ by
\begin{equation*}
\begin{aligned}
a\cdot h (y) = [p, (p \bs a(x)) (p\bs h(y))],
\end{aligned}
\end{equation*}
where $y \in F(M)$,  $x$ is the base point of the fiber of $y$
and $p$ is an arbitrary point in the fiber $P_x$.
\begin{proposition}
\label{fundamentalvectorfieldV}
Let $V$ be a $G$-module.
For every $a \in \Gamma(M, P\gtimes \glie)$ and every map $h \from F(M) \to
P\gtimes V $ of fiber bundles over $M$,
we have
$\rlie_{X^a} h = - a \cdot h$.
\end{proposition}
\begin{proof}
Let $x \in M$, $y \in F(M)_x$ and $p \in P_x$.
From Proposition~\ref{proprlieomegaflow}, it follows that
\begin{equation*}
\begin{aligned}
p\bs \rlie_{X^a} h (y) &=
 \vpr_V \left( \ddt0 ( \Phi_{X^a}^tp)\bs h (  y)\right).
\end{aligned}
\end{equation*}
Applying the definition of  $\Phi_{X^a}^t$, we get
\begin{equation*}
\begin{aligned}
p\bs \rlie_{X^a} h (y) &=
 \vpr_V \left( \ddt0  (p\exp(  (p\bs a(x))t))\bs h (  y)\right).
\end{aligned}
\end{equation*}
As for every  $g \in G$ and $z \in (P \gtimes V)_x$, we have $(pg) \bs z
=
g^{-1} ( p\bs z)$, the above formula can be rewritten in the form
\begin{equation*}
\begin{aligned}
p\bs \rlie_{X^a} h (y) &=
\vpr_V \left( \ddt0  \exp( - (p\bs a(x))t) (p\bs h (
y))\right).
\end{aligned}
\end{equation*}
By~\eqref{liealgebraactiononV}, we get
$p\bs \rlie_{X^a} h (y) =
 - (p\bs a(x)) (p\bs h (  y))$.
Hence $\rlie_{{X^a}} h  =-a \cdot h$.
\end{proof}
\medskip

When the target bundle is $P\gtimes\cg$ instead of $P \gtimes V$ the resulting
formula is more interesting.
To state the result, we need additional notation.
For $h \from F(M) \to P\gtimes {}_c G$, we deviate from the standard
conventions and denote by $h^{-1}$ the map from
$F(M)$ to $P\gtimes {}_c G$ given by $h^{-1} (y) = [ p, (p\bs h(y))^{-1}]$,
where $p\in P$ lies over the same point in $M$ as $y$.

Further, for $a \in \Gamma(M, P\gtimes \glie)$, we define $\Ad_h (a) \from F(M) \to
P\gtimes \glie$ by
\begin{equation*}
\begin{aligned}
\Ad_h(a) (y) = [p, \Ad_{ p \bs h(y)} (p \bs a(x))],
\end{aligned}
\end{equation*}
for all $x \in M$, $y\in F(M)_x$ and any $p \in P_x$.
\begin{proposition}
\label{fundamentalvectorfieldG}
For every $a \in \Gamma(M, P\gtimes \glie)$ and $h \from F(M) \to P\gtimes
{}_c G $,
we have
$\rlie_{X^a} h =  a\circ p_{M} - \Ad_{h^{-1}} (a)$,
where $p_M \from F(M) \to M$ is the bundle projection.
\end{proposition}
\begin{proof}
Write $*_c$ for the conjugation action of $G$ on itself, i.e.\ $h *_c g =
hgh^{-1}$ for all $h$, $g \in G$. Let $x \in M$, $p \in P_x$ and $y\in
F(M)_x$. By the same chain of arguments as in the proof
of Proposition~\ref{fundamentalvectorfieldV}, we get
\begin{equation*}
\begin{aligned}
p\bs \rlie_{{X}^a} h (y) &=
\omega_{MC} \left( \ddt0  \exp( - (p\bs a(x))t) *_c (p\bs h (  y))\right) .
\end{aligned}
\end{equation*}
For any $b\in \glie$ and $g \in G$, we have
$\exp (- bt) *_c g = \exp(-bt) g
\exp(bt)$.
Next
\begin{equation*}
\begin{aligned}
\omega_{MC}  \left( \ddt0 \exp(-bt)g \exp(bt) \right) & = \ddt0 g^{-1}
\exp(-bt) g \exp(bt) \\& = - \Ad_{g^{-1} }(b) + b.
\end{aligned}
\end{equation*}
Therefore with $b = p \bs a(x)$ and $g= p\bs h(y)$, we get
\begin{equation*}
\begin{aligned}
p\bs \rlie_{X^a} h (y) = - \Ad_{(p\bs h(y))^{-1}} ( p \bs a(x)) + (p \bs
a(x)).
\end{aligned}
\end{equation*}
This proves the proposition.
\end{proof}
More generally, for a $G$-equivariant $V$-valued $1$-form $\omega$ on $N$, we
define $*_\omega \from \glie \times N \to V$ by
\begin{equation}
\begin{aligned}
 b *_\omega z = \omega\left(\ddt0
\exp(bt)z\right).
\end{aligned}
\end{equation}
Then adapting the first part of  the proof
of Proposition~\ref{fundamentalvectorfieldV}, we get
\begin{equation}
\begin{aligned}
p \bs \rlie_{X^a}^{\omega} h(y) = - (p \bs a(x)) *_\omega (p \bs h(y)),
\end{aligned}
\end{equation}
for all $y \in F(M)$. This equation can be rewritten as
\begin{equation}
\label{rlieXageneric}
\begin{aligned}
\rlie_{X^a}^{\omega} h = - a \hat{*}_\omega h,
\end{aligned}
\end{equation}
for a suitably defined $\hat{*}_\omega$.

In the case $P$ is equipped with two $G$-principal connections, we can compare the
corresponding  covariant Darboux-Lie
derivatives $\lie^\cov_X$ and $\lie^{\widetilde{\cov}}_X$.
Write $X^H$
and
$X^{\widetilde{H}}$ for the horizontal lifts of the vector field $X \in \vf(M)$
to $P$ with
respect to these connections.  Then $X^H - X^{\widetilde{H}}$ is a vertical
vector field on $P$ and  there is a unique section $a\in
\Gamma(M, P\gtimes \glie)$ such that $X^H - X^{\widetilde{H}} = X^a$.
If $V$ is a $G$-module and $h\from F(M) \to P\gtimes V$ is a vector bundle
map,
Proposition~\ref{fundamentalvectorfieldV} implies that
\begin{equation}
\begin{aligned}
\lie^{\cov}_X h- \lie^{\widetilde{\cov}}_X h = - a \cdot h.
\end{aligned}
\end{equation}
If $h \from F(M) \to P\gtimes \cg$ is a fiber bundle map,
Proposition~\ref{fundamentalvectorfieldG} implies that
\begin{equation}
\begin{aligned}
\lie^{\cov}_X h- \lie^{\widetilde{\cov}}_X h = a\circ p_{M} -
\Ad_{h^{-1}} (a),
\end{aligned}
\end{equation}
where $p_{M} \from F(M) \to M$ is the projection of the fiber bundle
$F(M)$.

\section{Leibniz rule}\label{sec:leibniz}
In future articles,
we will need the expressions for
$\rlie_{\widetilde{X}} (s\cdot \beta)$, where $s \in
\Gamma(M, P\gtimes {}_c G)$ and $\beta \in \Omega^1(M, P\gtimes V)$, and for
$\rlie_{\widetilde{X}} ( \alpha\wedge \beta)$, where $\beta$ is as before and
$\alpha \in \Omega^1(M, P\gtimes \glie)$.
This section aims to explore the Leibniz rule for the Darboux-Lie derivative
with respect to some binary operations.

Notice that the map $s\cdot \beta$ can be described as the composition
\begin{equation*}
\begin{aligned}
TM \xrightarrow{\cong} M \times_M TM \xrightarrow{s\times_M \beta}  P\gtimes ( {}_c G \times V)
\xrightarrow{ } P \gtimes V
\end{aligned}
\end{equation*}
and $\alpha \wedge \beta$ as the composition
\begin{equation*}
\begin{aligned}
\Lambda^2 TM \to TM\otimes TM \xrightarrow{ \alpha \otimes \beta}  P\gtimes (\glie \otimes V) \to P
\gtimes V.
\end{aligned}
\end{equation*}
Thus in both cases we have a composition of
\begin{enumerate}[$\bullet$]
\item a natural map $\eta_M \from F(M) \to F_1(M) * F_2 (M)$, where
$*$ is either $\times$ or~$\otimes$;
\item a suitably defined  product $h_1 * h_2 \from F_1(M) * F_2(M) \to P \gtimes
(N_1 * N_2) $ of $h_1 \from F_1(M) \to P\gtimes N_1$ and $h_2 \from
F_2(M) \to P \gtimes N_2$;
\item the map from $P \gtimes (N_1 * N_2) \to P\gtimes N_3$
induced by a map $N_1 * N_2 \to N_3$.
\end{enumerate}
Below we develop a machinery to deal with Darboux-Lie derivatives of
compositions of the above type.

Applying Proposition~\ref{liederivativecomposition}
and~\eqref{liederivativeofnaturalmapiszero}, we conclude  that for every natural map
$\eta_M \from
F'(M) \to F(M)$ and for $h \from F(M) \to P\gtimes N$
\begin{equation*}
\begin{aligned}
\tilde\rlie_{\widetilde{X}} (h \circ \eta_M) = Th \circ
\trlift_{(\mathcal{F}'(X), \mathcal{F}(X))} \eta_{M} +
(\tilde\rlie_{\widetilde{X}} h )\circ \eta_M =
(\tilde\rlie_{\widetilde{X}} h) \circ \eta_M .
\end{aligned}
\end{equation*}
Hence
\begin{equation}
\label{darbouxlienaturalmap}
\begin{aligned}
\rlie_{\widetilde{X}}^{\omega} (h \circ \eta_M) =  (
\rlie_{\widetilde{X}}^\omega h) \circ \eta_M.
\end{aligned}
\end{equation}

\subsection{Darboux-Lie derivative of a product}
Suppose $F_1$ and $F_2$ are natural bundles. Define the natural bundle
$F_1 \times F_2$ by $(F_1 \times F_2) (M) = F_1(M) \times_M F_2(M)$.

Given $h_1 \from F_1(M) \to P\gtimes N_1$ and $h_2 \from F_2(M) \to P\gtimes
N_2$, define
\begin{equation*}
\begin{aligned}
h_1 \times_M h_2 \from (F_1 \times F_2) (M) & \to P\gtimes (N_1 \times N_2) \\
(y_1, y_2) & \mapsto [p, (p \bs h_1(y_1), p \bs h_2(y_2) )],
\end{aligned}
\end{equation*}
where $y_1\in F_1(M)_x$, $y_2\in F_2(M)_x$, $p \in P_x$ for some $x \in M$.
Let $\omega_1 \from TN_1 \to V_1$ and $\omega_2 \from TN_2 \to V_2$ be
$G$-equivariant $1$-forms with values in $G$-modules $V_1$ and $V_2$.
Define $\omega_1 \times \omega_2 \from T(N_1 \times N_2) \to V_1 \oplus V_2$ by
\begin{equation*}
\begin{aligned}
\omega_1 \times \omega_2 \left( \ddt0 (\gamma_1(t), \gamma_2(t)) \right) =
\left( \omega_1 ( \dot\gamma_1(0)), \omega_2 ( \dot\gamma_2 (0)) \right).
\end{aligned}
\end{equation*}
By Proposition~\ref{proprlieomegaflow}, we get for $\widetilde{X} \in \vf(P)^G$ over
$X \in \vf(M)$
\begin{equation*}
\begin{aligned}
p \bs\!\! &\,\, \rlie_{\widetilde{X}}^{\omega_1 \times \omega_2} (h_1 \times_M h_2) (y_1,
y_2)= \\& =
\omega_1 \times \omega_2 \left( \ddt0 \Phi_{\widetilde{X}}^t p \bs (h_1 \times_M h_2) ( \Phi_{\mathcal{F}_1(X)}^t y_1, \Phi_{\mathcal{F}_2(X)}^t y_2) \right) \\& =
\omega_1 \times \omega_2 \left( \ddt0 \big(\, \Phi_{\widetilde{X}}^t p \bs h_1
\circ \Phi_{\mathcal{F}_1(X)}^t (y_1)\, ,\,  \Phi_{\widetilde{X}}^t p \bs  h_2 \circ \Phi_{\mathcal{F}_2(X)}^t (y_2)\, \big) \right) \\& =
\left( \omega_1\!  \left( \ddt0\!\!  \Phi_{\widetilde{X}}^t p \bs h_1 \circ
\Phi_{\mathcal{F}_1(X)}^t (y_1)\right) ,\, \omega_2 \!\left( \ddt0\!\!
\Phi_{\widetilde{X}}^t p \bs  h_2 \circ \Phi_{\mathcal{F}_2(X)}^t (y_2)\right) \right) \\& =
\left( p \bs \rlie_{\widetilde{X}}^{\omega_1} h_1 (y_1),
p \bs \rlie_{\widetilde{X}}^{\omega_2} h_2 (y_2) \right).
\end{aligned}
\end{equation*}
Hence
\begin{equation}
\label{rlieproduct}
\begin{aligned}
\rlie_{\widetilde{X}}^{\omega_1 \times \omega_2} (h_1 \times_M h_2) =
\rlie_{\widetilde{X}}^{\omega_1} h_1 \times_M \rlie_{\widetilde{X}}^{\omega_2}
h_2.
\end{aligned}
\end{equation}
\subsection{Darboux-Lie derivative of a tensor product}
For natural vector bundles $E_1$, $E_2$, we define the natural vector bundle
$E_1\otimes E_2$ by $(E_1 \otimes E_2)(M) = E_1 (M) \otimes E_2(M)$.
Given two vector bundle maps $h_1 \from E_1(M) \to P\gtimes V_1$ and
$h_2 \from E_2(M) \to P\gtimes V_2$, we define
\begin{equation*}
\begin{aligned}
h_1 \otimes h_2 \from (E_1 \otimes E_2) (M) & \to P \gtimes (V_1 \otimes
V_2) \\
y_1\otimes y_2 & \mapsto [p, (p \bs h_1(y_1))\otimes (p \bs h_2(y_2) )],
\end{aligned}
\end{equation*}
where $y_1\in E_1(M)_x$, $y_2\in E_2(M)_x$, $p \in P_x$ for some $x \in M$.
We have for all $x \in M$, $y_1 \in E_1(M)_x$, $y_2 \in E_2(M)_x$
\begin{equation*}
\begin{aligned}
\Phi^t_{(\mathcal{E}_1 \otimes \mathcal{E}_2)(X)} (y_1 \otimes y_2) =
\Phi^t_{\mathcal{E}_1(X)} (y_1) \otimes \Phi^t_{\mathcal{E}_2(X)}(y_2).
\end{aligned}
\end{equation*}
By Proposition~\ref{proprlieomegaflow}, we get for $\widetilde{X} \in \vf(P)^G$ over
$X \in \vf(M)$
\begin{equation*}
\begin{aligned}
p \bs\!\! &\,\, \rlie_{\widetilde{X}} (h_1 \otimes h_2) (y_1 \otimes
y_2)= \\& =
 \vpr_{V_1 \otimes V_2} \left( \ddt0 \Phi_{\widetilde{X}}^t p \bs
(h_1 \otimes h_2) ( \Phi_{\mathcal{F}_1(X)}^t y_1\otimes \Phi_{\mathcal{F}_2(X)}^t y_2) \right) \\& =
  \vpr_{V_1 \otimes V_2}\left( \ddt0  \Phi_{\widetilde{X}}^t p \bs h_1
\circ \Phi_{\mathcal{F}_1(X)}^t (y_1)\otimes  \Phi_{\widetilde{X}}^t p \bs  h_2
\circ \Phi_{\mathcal{F}_2(X)}^t (y_2)\,  \right).
\end{aligned}
\end{equation*}
Given curves $\gamma_1 \from I \to V_1$ and $\gamma_2 \from I \to V_2$, we have
\begin{equation*}
\begin{aligned}
 \vpr_{V_1 \otimes V_2} \left( \ddt0 \gamma_1(t) \otimes
\gamma_2(t) \right) & =
 \vpr_{V_1} ( \dot \gamma_1(0)) \otimes \gamma_2(0) \\&
\phantom{=}+
\gamma_1(0) \otimes  \vpr_{V_2} ( \dot \gamma_2(0)).
\end{aligned}
\end{equation*}
Therefore
\begin{equation}
\label{beforerlietensorproduct}
\begin{aligned}
p \bs \rlie_{\widetilde{X}} (h_1 \otimes h_2) (y_1 \otimes
y_2) & =
(p \bs  \rlie_{\widetilde{X}} h_1 (y_1) ) \otimes ( p \bs h_2(y_2))
\\ & \phantom{=} + ( p \bs h_1(y_1)) \otimes
(p \bs \rlie_{\widetilde{X}} h_2 (y_2)).
\end{aligned}
\end{equation}
Thence
\begin{equation}
\label{rlietensorproduct}
\begin{aligned}
\rlie_{\widetilde{X}} (h_1 \otimes h_2) = \rlie_{\widetilde{X}} h_1 \otimes h_2
+ h_1
\otimes \rlie_{\widetilde{X}} h_2.
\end{aligned}
\end{equation}

\subsection{Darboux-Lie derivative of a composition}
In this subsection we deal with the Darboux-Lie derivative of $(\id \gtimes f)
\circ h$, where $h \from
F(M) \to P\gtimes N$ and $f$ is a $G$-equivariant map from $N$ to $N'$.
Suppose $\omega' \from TN'\to
V'$ is a $G$-equivariant $1$-form with values in a $G$-module $V'$. By
Proposition~\ref{proprlieomegaflow},
for $h \from F(M) \to P\gtimes N$ and every $x \in M$, $y \in F(M)_x$ and $p \in P_x$
\begin{equation*}
\begin{aligned}
p \bs \rlie_{\widetilde{X}}^{\omega'} ( (\id \gtimes f) \circ h)(y) & =\omega'\left( \ddt0
  \Phi_{\widetilde{X}}^t p \bs (( \id \gtimes f) \circ h \circ
\Phi_{\mathcal{F}(X)}^t (y))
\right) \\& = f^* \omega' \left( \ddt0 \Phi_{\widetilde{X}}^t p \bs (h \circ
\Phi_{\mathcal{F}(X)}^t (y))
\right) \\& =p \bs \rlie_{\widetilde{X}}^{f^*\omega'} h (y).
\end{aligned}
\end{equation*}
Hence
\begin{equation}
\label{postcomposition}
\begin{aligned}
\rlie^{\omega'}_{\widetilde{X}} ( (\id \gtimes f) \circ h) =
\rlie_{\widetilde{X}}^{f^* \omega'} h.
\end{aligned}
\end{equation}
The ability to write the above formula is the reason  we don't require $\omega$
to be a parallelization form in the definition of
$\rlie^\omega_{\widetilde{X}}$, since even if $\omega'$ is a parallelization
form its pull-back $f^*\omega'$, in general, is not.

Formula~\eqref{postcomposition} becomes more useful if there is a relation between $f^*\omega'$
and
a $G$-equivariant $1$-form $\omega\from TN \to V$, where $V$ is
a
$G$-module.
The vector space $\Hom_\R(V,V')$ has the $G$-module structure defined by
$(g\alpha)(v) = g \alpha( g^{-1} v)$
for all $g \in G$, $\alpha \in \Hom_\R(V,V')$ and $v \in V$.
 Now, assume that there is a $G$-equivariant map $\varphi \from  N \to
\Hom_\R(V,V')$ such that
\begin{equation}
\label{varphiproperty}
\begin{aligned}
(f^*\omega')_z = \varphi(z) \circ \omega_z
\end{aligned}
\end{equation}
for each
$z\in N$.
Let $x \in M$, $y\in F(M)_x$ and $p \in P_x$. Denote the curve $\Phi_{\widetilde{X}}^t p \bs (h \circ
\Phi_{\mathcal{F}(X)}^t (y)$ by $\gamma$.
By Proposition~\ref{proprlieomegaflow}, we get
\begin{equation}
\label{pbsrliewtXfstaromegahy}
\begin{aligned}
p \bs \rlie_{\widetilde{X}}^{f^*\omega'} h (y)
= f^*\omega'(\dot\gamma(0))
= \varphi (p\bs h(y)) ( \omega(\dot\gamma(0))) =
  \varphi(p\bs h(y))  ( p \bs \rlie_{\widetilde{X}}^\omega h (y)).
\end{aligned}
\end{equation}
Hence
\begin{equation}
\label{chainrulewithvarphi}
\begin{aligned}
\rlie_{\widetilde{X}}^{\omega'} ( (\id\gtimes f)\circ h) (y)= \rlie_{\widetilde{X}}^{f^*\omega'}
h (y)= (( \id \gtimes \varphi)\, h(y))\, ( \rlie_{X}^{\omega} h (y)),
\end{aligned}
\end{equation}
where we define an action of $\Gamma(M, P\gtimes \Hom_{\R}(V,V'))$
on $\Gamma(M, P\gtimes V)$ by
\begin{equation*}
\begin{aligned}
(\psi(s)) (x) = [ p , (p \bs \psi(x)) ( p \bs s(x)) ],
\end{aligned}
\end{equation*}
for all $\psi \in \Gamma(M, P\gtimes \Hom_{\R}(V,V'))$, $s \in \Gamma(M,
P\gtimes V)$, $x \in M$ and $p \in P_x$.

Next we give several examples of $N$, $N'$, $\omega$, $\omega'$ and $f$ for
which there exists a map $\varphi$ satisfying~\eqref{varphiproperty}.
\begin{example}
Take $N=V$, $N' = V'$, $f \from V \to V'$  a
linear $G$-equivariant map,
 $\omega =  \vpr_{V}$ and $\omega' =
 \vpr_{V'}$.
In this case $f^* \omega' = f\circ \omega$.
Hence  $\varphi \from V \to \Hom_\R(V,V')$ is the constant map that
sends each $v \in V$ into $f$ and~\eqref{chainrulewithvarphi} becomes
\begin{equation}
\label{chainrulewithoutvarphi}
\begin{aligned}
\rlie_{\widetilde{X}} ( ( \id\gtimes f) \circ h) (y) = ( \id\gtimes f) (
\rlie_{\widetilde{X}} h(y)).
\end{aligned}
\end{equation}
\end{example}

\begin{example}
Take $N= G\times V$, $N' =V$ and $f \from G\times
V \to V$ to be the
 action of $G$ on $V$.
We equip $N$ with the form $\omega := \omega_{MC} \times  \vpr_V$
and
$N'$ with the form $\omega' :=  \vpr_V$.
Fix $g \in G$ and $v \in V$. Every element of $T_{(g,v)} (G\times V)$ is of the
form $\ddt0 (g\exp(at), v + ut) $ for suitable $a \in \glie$ and $u \in V$.
We get
\begin{equation}
\label{fstaromegaprimegeat}
\begin{aligned}
f^* \omega'  \left( \ddt0 ( ge^{at}, v + ut) \right) &=
\vpr_V \left( \ddt0 \left(ge^{at} v + ge^{at} ut \right) \right) \\&
= g(av) + gu = g( av +u).
\end{aligned}
\end{equation}
Notice that
\begin{equation*}
\begin{aligned}
a = \ddt0 \exp(at) = (T_e L_g)^{-1} \ddt0 g \exp (at) = \omega_{MC}\left( \ddt0
g\exp(at)
\right)
\end{aligned}
\end{equation*}
and
\begin{equation*}
\begin{aligned}
u =  \vpr_V \left(\ddt0 (v + ut)  \right) .
\end{aligned}
\end{equation*}
Hence
\begin{equation}
\label{omegamctimespr}
\begin{aligned}
(\omega_{MC} \times \vpr_V) \left( \ddt0 ( ge^{at}, v + ut)
\right) & = \left(  a,u \right).
\end{aligned}
\end{equation}
Now~\eqref{fstaromegaprimegeat} and \eqref{omegamctimespr} imply that
for all $(g,v) \in
G\times V$
\begin{equation*}
\begin{aligned}
(f^* \omega')_{(g,v)}  = \varphi(g,v) \circ
((\omega_{MC})_g \times (\vpr_V)_v ) ,
\end{aligned}
\end{equation*}
 where
\begin{align}
\label{varphiforGtimesV}
\varphi \from G \times V & \to \Hom_{\R} (\glie \oplus V, V) \\
\nonumber
(g,v) & \mapsto (  ( a, u)  \mapsto g ( av + u)).
\end{align}
\end{example}

\begin{example}
Let $H$ be a Lie group equipped with a left $G$-action by Lie automorphisms.
Take $N=H\times H$, $N'= H$ and $f\from H\times H \to H$ the multiplication
map.
Write $\omega_H$ for the Maurer-Cartan form on $H$. It is $G$-equivariant,
since $G$ acts on $H$ by automorphisms.
The Maurer-Cartan form on the direct product $H\times H$ coincides with
$\omega_{H} \times
\omega_{H}$.

For $h_1$, $h_2 \in H$ and $a_1$, $a_2 \in \hlie = T_e H$, we have
\begin{equation*}
\begin{aligned}
f^* \omega_H  \left( \ddt0 (h_1 e^{a_1t}, h_2 e^{a_2t}) \right)
&  = \omega_H \left( \ddt0 h_1 h_2 h_2^{-1} e^{a_1t} h_2 e^{a_2t} \right)
\\& = \omega_H \left( \ddt0 h_1 h_2 \exp ( \Ad_{h_2^{-1}}(a_1t) )  e^{a_2 t}
\right)
\\& =  \ddt0 \exp ( \Ad_{h_2^{-1}}(a_1) t)  \exp(a_2 t)
\\& = \Ad_{h_2^{-1}}(a_1) + a_2.
\end{aligned}
\end{equation*}
Hence $(f^* \omega_H)_{h_1,h_2} = \varphi(h_1,h_2) \circ ( \omega_H \times
\omega_H)_{h_1,h_2}$, where
\begin{align}
\label{varphiprpduct}
\varphi  \from H\times H & \to \Hom_\R( \hlie \oplus \hlie, \hlie) \\
\nonumber
(h_1, h_2) & \mapsto ( (a_1,a_2) \mapsto \Ad_{h_2^{-1}} (a_1) + a_2 ).
\end{align}\end{example}
\subsection{Synthesis}
The aim of this subsection is to obtain rather general Lie-type formulas by
  combining the results already obtained
in this section. We will treat the special cases in examples.

Let $F_1$, $F_2$ be natural  bundles  and $h_1 \from F_1(M) \to
P\gtimes N_1$, $h_2 \from F_2(M) \to P\gtimes N_2$ morphisms of fibred
manifolds over $M$.
Suppose  $\omega_1 \from TN_1 \to V_1$, $\omega_2 \from
TN_2 \to V_2$, $\omega \from TN \to V$ are $G$-equivariant forms, $f \from N_1
\times N_2
\to N$ is a $G$-equivariant map, and $\varphi \from N_1 \times N_2 \to \Hom_{\R}
( V_1 \oplus V_2, V) $  is a smooth map
such that
\begin{equation*}
\begin{aligned}
(f^* \omega)_z = \varphi(z) \circ ( \omega_1 \times \omega_2)_z.
\end{aligned}
\end{equation*}
Let $\eta  \from F \to F_1 \times F_2 $ be a natural bundle map.
Define
\begin{equation*}
\begin{aligned}
h_1 \times_{f,\eta} h_2 \from F(M) & \to P\gtimes N \\
\end{aligned}
\end{equation*}
to be the composition $(\id \gtimes f)\circ (h_1 \times_M h_2 ) \circ \eta_M$.

Fix $y \in F(M)$. Then $\eta_M(y) = (y_1, y_2)$ for suitable   $y_1 \in F_1(M)$,
$y_2 \in F_2(M)$. Using~\eqref{darbouxlienaturalmap} and~\eqref{chainrulewithvarphi}, we get
\begin{equation}
\label{rlietildeXomegahtimesfetahy}
\begin{aligned}
\rlie_{\widetilde{X}}^{\omega} (h_1 \times_{f,\eta} h_2) (y) & =
\rlie_{\widetilde{X}}^{\omega} ( (\id\gtimes f)
\circ (h_1 \times_M h_2) \circ \eta_M) (y)
\\ & = \rlie_{\widetilde{X}}^\omega ( (\id\gtimes f)
\circ (h_1 \times_M h_2)) ( y_{1} , y_{2})
\\& = ( \id\gtimes \varphi) ( (h_1 \times_M h_2) ( y_{1} , y_{2})) (
\rlie_{\widetilde{X}}^{\omega_1 \times \omega_2} ( h_1 \times_M h_2) ( y_{1} , y_{2})).
\end{aligned}
\end{equation}
By definition of $h_1 \times_M h_2$, we have
\begin{equation*}
\begin{aligned}
(h_1 \times_M h_2) ( y_{1} , y_{2}) = [ p , (p \bs h_1 (y_{1}),   p \bs h_2
(y_{2})) ].
\end{aligned}
\end{equation*}
Hence
\begin{equation*}
\begin{aligned}
( \id\gtimes \varphi) ( (h_1 \times_M h_2) ( y_{1} , y_{2}) )
 =  [ p ,  \varphi (p \bs h_1 (y_{1})  , p \bs h_2 (y_{2}) )  ].
\end{aligned}
\end{equation*}
By \eqref{rlieproduct}, we have
\begin{equation*}
\begin{aligned}
\rlie_{\widetilde{X}}^{\omega_1 \times \omega_2} (h_1 \times_M h_2) (y_{1} , y_{2}) =
[ p , ( p \bs \rlie_{\widetilde{X}}^{\omega_1} h_1 (y_{1})  ,
  p \bs \rlie_{\widetilde{X}}^{\omega_2} h_2 (y_{2}) ) ].
\end{aligned}
\end{equation*}
Thence~\eqref{rlietildeXomegahtimesfetahy} becomes
\begin{gather}
\label{verygeneralleibniztimesM}
\begin{aligned}
\rlie_{\widetilde{X}}^\omega (h_1 \times_{f,\eta} h_2) (y) =
[ p , \varphi & (p \bs h_1 (y_{1})  , p \bs h_2 (y_{2}) ) \\& ( p \bs \rlie_{\widetilde{X}}^{\omega_1} h_1 (y_{1})  ,
  p \bs \rlie_{\widetilde{X}}^{\omega_2} h_2 (y_{2})  )].
\end{aligned}
\end{gather}
\begin{example}
Suppose $s \in \Gamma(M, P \gtimes {}_c G)$ and $\beta \from F(M) \to P \gtimes
V$, where $V$ is a $G$-module. We define $\eta_M \from F(M) \to M\times_M
F(M)$ to be $Z \mapsto (x, Z)$ for
every $x \in M$ and $Z \in F(M)_x$.
Write $f$ for the action map $G \times V \to V$.
Then $s\cdot \beta$ defined at the beginning of the section coincides with
$s\times_{f,\eta} \beta$.
Taking $\varphi$ defined by~\eqref{varphiforGtimesV} and
applying~\eqref{verygeneralleibniztimesM}, we get for $x \in M$, $p\in P_x$ and  $Z \in T_x M$
\begin{equation*}
\begin{aligned}
p \bs \rlie_{\widetilde{X}} (s\cdot \beta) (Z) & = \varphi(p \bs s(x), p \bs
\beta(Z) ) ( p \bs \rlie_{\widetilde{X}} s(x) , p \bs \rlie_{\widetilde{X}} \beta (Z) )
\\& =
(p \bs s(x)) \left(  (p \bs \rlie_{\widetilde{X}} s(x) ) ( p \bs \beta(Z)) + p
\bs \rlie_{\widetilde{X}} \beta (Z)  \right).
\end{aligned}
\end{equation*}
Hence
\begin{equation}
\label{rlietildeXsbeta}
\begin{aligned}
\rlie_{\widetilde{X}}( s\cdot \beta) = s \cdot (
\rlie_{\widetilde{X}} s \cdot \beta + \rlie_{\widetilde{X}}
\beta ).
\end{aligned}
\end{equation}
If $P$ is equipped with a $G$-principal connection and $X \in \vf(M)$, then we get the
following property for the covariant Darboux-Lie derivative
\begin{equation}
\begin{aligned}
\lie^\cov_X ( s\cdot \beta) & = s \cdot ( \lie^\cov_X s \cdot \beta + \lie^\cov_X
\beta) .
\end{aligned}
\end{equation}
\end{example}
\begin{example}
Now let $s_1$, $s_2 \in \Gamma(M, P\gtimes {}_c G)$. Notice that the
multiplication map $\mu_G \from {}_c G \times {}_c G \to {}_c G$ is
$G$-equivariant.
Define $\eta_M \from M \to M \times_M M$ by $\eta_M(x) = (x,x)$.
We write $s_1 \cdot s_2$ for $s_1 \times_{\mu_G,\eta} s_2$.
Taking $\varphi$ defined by~\eqref{varphiprpduct} and
applying~\eqref{verygeneralleibniztimesM}, we get
for all $x \in M$ and $p \in P_x$,
\begin{equation*}
\begin{aligned}
p \bs \rlie_{\widetilde{X}} (s_1\cdot s_2) (x) & = \varphi ( p \bs s_1(x), p \bs
s_2(x)) (p \bs \rlie_{\widetilde{X}} s_1 (x), p \bs \rlie_{\widetilde{X}} s_2
(x))
\\& = \Ad_{ (p \bs s_2(x))^{-1}} \left( p \bs \rlie_{\widetilde{X}} s_1 (x)
\right) + (p \bs \rlie_{\widetilde{X}} s_2 (x)) .
\end{aligned}
\end{equation*}
Thus with appropriate definitions for $s_2^{-1}$ and $\Ad_{s_2^{-1}}$, we get
\begin{equation}
\label{AdsinvrlietildeXs}
\begin{aligned}
\rlie_{\widetilde{X}} (s_1\cdot s_2) =
\Ad_{ s_2^{-1}} \left( \rlie_{\widetilde{X}} s_1
\right) +  \rlie_{\widetilde{X}} s_2.
\end{aligned}
\end{equation}
In the case $P$ is equipped with a $G$-principal  connection, this implies
\begin{equation}
\begin{aligned}
\lie^\cov_X (s_1\cdot s_2) &= \Ad_{s_2^{-1}} ( \lie^\cov_X s_1) + \lie^\cov_X
s_2.
\end{aligned}
\end{equation}
\end{example}
\medskip
\medskip

Now we turn our attention to the Leibniz rule involving tensor product.
Let $E_1$, $E_2$ be natural vector bundles and $h_1 \from E_1(M) \to
P\gtimes V_1$, $h_2 \from E_2(M) \to P\gtimes V_2$ morphisms of vector bundles,
$f \from V_1 \otimes V_2 \to V$ a homomorphism of $G$-modules, and $\eta \from E \to E_1 \otimes E_2 $ a natural transformation of natural  vector bundles.
Write $h_1 \otimes_{f,\eta} h_2 $ for the composition $(\id \gtimes f) \circ (h_1
\otimes  h_2)
\circ \eta_M$.

Fix $x \in M$, $y \in E(M)_x$, and $p \in P_x$.
Then $\eta_M(y)$ can be written as $\sum_{i\in I} y_{1,i} \otimes y_{2,i}$ for
suitable elements $y_{1,i} \in E_1(M)_x$ and $y_{2,i} \in E_2(M)_x$.
Using~\eqref{darbouxlienaturalmap} and~\eqref{chainrulewithoutvarphi}, we get
\begin{equation*}
\begin{aligned}
\rlie_{\widetilde{X}} (h_1 \otimes_{f,\eta} h_2) (y) & =
\rlie_{\widetilde{X}} ( (\id \gtimes f)
\circ (h_1 \otimes h_2) \circ \eta_M) (y)
\\ & = \rlie_{\widetilde{X}} ( (\id \gtimes f) \circ (h_1 \otimes h_2))
\left(\eta_M(y)\right)
\\ & = (\id \gtimes f) \left( \rlie_{\widetilde{X}} (h_1 \otimes h_2)
\left(\eta_M(y)\right) \right).
\end{aligned}
\end{equation*}
Applying~\eqref{beforerlietensorproduct}, we get
\begin{equation*}
\begin{aligned}
p \bs \rlie_{\widetilde{X}} (h_1 \otimes h_2)
\left(\eta_M(y)\right) & = \sum_{i \in I} (
p \bs \rlie_{\widetilde{X}} h_1 (y_{1,i})
\otimes p \bs h_2
(y_{2,i}) \\ & \phantom{=\sum_{i\in I}} + p \bs h_1(y_{1,i}) \otimes p \bs\rlie_{\widetilde{X}} h_2
(y_{2,i})) .
\end{aligned}
\end{equation*}
Thence
\begin{equation}
\label{verygeneralleibnizotimes}
\begin{aligned}
\rlie_{\widetilde{X}} (h_1 \otimes_{f,\eta} h_2) (y) & =
\sum_{i \in I}
[ p, f(p \bs \rlie_{\widetilde{X}} h_1 (y_{1,i})
\otimes p \bs h_2 (y_{2,i}))]
 \\ & \phantom{=\sum_{i\in I}} + [ p, f(p \bs h_1(y_{1,i}) \otimes p \bs\rlie_{\widetilde{X}} h_2
(y_{2,i})) ].
\end{aligned}
\end{equation}
\begin{example}
Let $\alpha \in \Omega^1(M, P \gtimes \glie)$ and $\beta \in \Omega^1(M,
P\gtimes V)$, where $V$ is a $G$-module.
Define $\alpha \wedge \beta \in \Omega^2(M, P \gtimes V)$ by
\begin{equation*}
\begin{aligned}
(\alpha \wedge \beta) (Z_1, Z_2) = \alpha (Z_1) \cdot \beta(Z_2) -
\alpha(Z_2) \cdot \beta(Z_1).
\end{aligned}
\end{equation*}
Alternatively, it can be written as the composition $(\id \gtimes f) \circ
( \alpha \otimes \beta) \circ \eta_M$, where
\begin{equation*}
\begin{aligned}
\eta_M \from \Lambda^2 TM & \to TM \otimes TM \\
 Z_1 \wedge Z_2 & \mapsto Z_1 \otimes Z_2 - Z_2 \otimes Z_1
\end{aligned}
\end{equation*}
 and $f \from \glie \otimes V \to V$ is the action of $\glie$ on $V$.
Notice that $f$ is $G$-equivariant.
Thence $\alpha \wedge \beta = \alpha\otimes_{f,\eta} \beta$.
Applying~\eqref{verygeneralleibnizotimes}, we get
\begin{equation*}
\begin{aligned}
p \bs \rlie_{\widetilde{X}} (\alpha \wedge \beta) (Z_1, Z_2) & =
( p \bs \rlie_{\widetilde{X}} \alpha (Z_1)) ( p \bs  \beta (Z_2))
+ ( p \bs \alpha (Z_1)) ( p \bs \rlie_{\widetilde{X}} \beta (Z_2)) \\& \phantom{=}
- ( p \bs \rlie_{\widetilde{X}} \alpha (Z_2)) ( p \bs \beta (Z_1))
- ( p \bs \alpha (Z_2)) ( p \bs \rlie_{\widetilde{X}} \beta (Z_1)).
\end{aligned}
\end{equation*}
Hence
\begin{equation}
\label{leibnizforwedgeproduct}
\begin{aligned}
\rlie_{\widetilde{X}} (\alpha \wedge \beta) = \rlie_{\widetilde{X}} \alpha
\wedge \beta + \alpha \wedge
\rlie_{\widetilde{X}} \beta.
\end{aligned}
\end{equation}
\begin{remark}
More generally we can take $\alpha \in \Omega^k(M, P\gtimes \glie)$ and
$\beta \in \Omega^\ell(M, P\gtimes V)$. Then $\alpha \wedge \beta$ can be
defined as $\alpha \otimes_{f,\eta} \beta$, where $f$ is the same as above and
\begin{equation*}
\begin{aligned}
\eta_M \from  \Lambda^{k + \ell} TM&\to \Lambda^k TM \otimes \Lambda^\ell TM\\
Z_1 \wedge \dots
\wedge Z_{k+\ell} & \mapsto \sum_{\sigma\in \shuffles_{k,\ell}}
(-1)^{\sigma} Z_{\sigma(1)} \wedge \dots \wedge Z_{\sigma(k)} \otimes
Z_{\sigma(k+1)} \wedge \dots \wedge Z_{\sigma(k+\ell)},
\end{aligned}
\end{equation*}
where $\shuffles_{k,l}$ is the set of all $(k,\ell)$-shuffles.
As expected, for every $\widetilde{X} \in \vf(P)^G$ one gets
\begin{equation}
\label{rliealphabetahigerdegrees}
\begin{aligned}
\rlie_{\widetilde{X}} (\alpha \wedge \beta) = \rlie_{\widetilde{X}} \alpha \wedge \beta + \alpha \wedge
\rlie_{\widetilde{X}} \beta.
\end{aligned}
\end{equation}
In the case $P$ is equipped with a $G$-principal connection and $X \in \vf(M)$, this
implies
\begin{equation}
\lie^\cov_X (\alpha \wedge \beta)  = \lie^\cov_X \alpha \wedge \beta + \alpha
\wedge \lie^\cov_X \beta.
\end{equation}
\end{remark}
\end{example}

\section{Cartan magic formula}\label{sec:magic}
Suppose $\cov$ is a covariant derivative on a vector bundle $E$ Then the
exterior covariant derivative $d^\cov$ of
$\beta \in \Omega^k(M,E)$  is given by
\begin{equation*}
\begin{aligned}
d^\cov \beta (X_0,\dots,X_{k}) & := \sum_{j=0}^k (-1)^{j} \cov_{X_j}
( \beta(X_0,\dots, \widehat{X}_j,\dots, X_{k}) )
\\ & \phantom{=} + \sum_{ i<j} (-1)^{i+j} \beta ( [X_i,X_j], X_0,\dots,
\widehat{X}_i,\dots,
\widehat{X}_j,\dots,X_{k}).
\end{aligned}
\end{equation*}
In the special case when $\cov$ is a covariant derivative on $P\gtimes V$
induced by a $G$-principal connection on $P$, we get
\begin{equation*}
\begin{aligned}
\cov_{X_j} ( \beta (X_0,\dots, \widehat{X}_j,\dots, X_{k})) = \lie^\cov_{X_j} (
\beta (X_0,\dots, \widehat{X}_j,\dots, X_{k})).
\end{aligned}
\end{equation*}
Indeed, let $s \in \Gamma(M, P\gtimes V)$ and $X^H$ the horizontal lift of
$X$ to $P$. Then $X^H_V$ it the $\cov$-horizontal lift of $X$ to $P\gtimes V$
and
by the definition of covariant Darboux-Lie derivative
\begin{equation*}
\begin{aligned}
\lie^\cov_X(s)(x) & =  \lie_{(X, X^H_V)}  (s)(x) =
\vpr_V\left( \ddt{t=0} \Phi_{X^H_V}^{-t} \circ s\circ \Phi_X^t (x) \right)
 = (\cov_X s)(x) .
\end{aligned}
\end{equation*}
Hence
\begin{equation}
\label{dcovbeta}
\begin{aligned}
d^\cov \beta (X_0,& \dots,X_{k})  = \sum_{j=0}^k (-1)^{j} \lie^\cov_{X_j}
( \beta(X_0,\dots, \widehat{X}_j,\dots, X_{k}) )
\\ & \phantom{=} + \sum_{i<j} (-1)^{i+j} \beta ( [X_i,X_j], X_0,\dots,
\widehat{X}_i,\dots,
\widehat{X}_j,\dots,X_{k}).
\end{aligned}
\end{equation}

Let $Y_1$,\dots, $Y_k$ and $Z$ be vector fields on $M$.
The section $\beta(Y_1,\dots, Y_k)$ of $P\gtimes V$ is the composition of
$\opwedge\limits_{j=1}^k Y_j \from M \to \wedge^k TM$
and $\beta \from \wedge^k TM \to P\gtimes V$.
Write $\widetilde{Z}$ for the canonical lift of $Z$ to $\wedge^k TM$.
Applying Proposition~\ref{liederivativecomposition}, we get
\begin{equation*}
\begin{aligned}
\tilde\rlie_{Z^H} ( \beta \circ \opwedge\limits_{j=1}^k Y_j) =
T\beta \circ \trlift_{(Z,\widetilde{Z})} (\opwedge\limits_{j=1}^k Y_j)
+ ( \trlift_{(\widetilde{Z},Z^H_V)}\beta) \circ \opwedge_{j=1}^k Y_j.
\end{aligned}
\end{equation*}

Applying $\id \gtimes \vpr_V$ to the both sides, we obtain
\begin{equation*}
\begin{aligned}
\lie^{\cov}_{Z} ( \beta \circ \opwedge\limits_{j=1}^k Y_j) =
\beta \circ \lie_{(Z,\widetilde{Z})} (\opwedge\limits_{j=1}^k Y_j)
+ ( \lie_{(\widetilde{Z},Z^H_V)}\beta) \circ \opwedge_{j=1}^k Y_j.
\end{aligned}
\end{equation*}
Now $\lie_{(Z,\widetilde{Z})}$ is the usual Lie derivative $\lie_{Z}$ of tensor fields and
$\lie_{(\widetilde{Z}, Z^H_V)}$ is the covariant Darboux-Lie derivative
$\lie^\cov_{Z}$.
Hence
\begin{equation}
\label{rlieZH}
\begin{aligned}
\lie^{\cov}_{Z} ( \beta &\circ \opwedge\limits_{j=1}^k Y_j) =
\beta \circ \lie_{Z} (\opwedge\limits_{j=1}^k Y_j)
+ ( \lie^{\cov}_{Z}\beta) \circ \opwedge_{j=1}^k Y_j
\\ & = \sum_{j=1}^k \beta ( Y_1,\dots, [Z,Y_j],\dots,Y_k) +
\lie^{\cov}_{Z} \beta ( Y_1,\dots, Y_k).
\end{aligned}
\end{equation}
Taking $X_0=Z$ and $X_i = Y_i$ for $1\le i\le k$ in~\eqref{dcovbeta}, we get
\begin{equation*}
\begin{aligned}
i_Z d^\cov \beta ( Y_1,\dots,Y_k) & =
\lie^{\cov}_{Z} (\beta(Y_1,\dots,Y_k) ) - \sum_{j=1}^{k} (-1)^{j-1} \lie^{\cov}_{Y_j}
(\beta(Z,Y_1,\dots, \widehat{Y}_j,\dots,Y_k))
\\ & \phantom{=} +  \sum_{j=1}^k (-1)^j \beta( [Z,Y_j], Y_1,\dots,
\widehat{Y}_j,\dots,k)
\\ & \phantom{=} - \sum_{i<j} (-1)^{i+j} \beta(Z, [Y_i,Y_j], Y_1, \dots,
\widehat{Y}_i,\dots, \widehat{Y}_j,\dots, Y_{k})
\\ & = (\lie^{\cov}_{Z} \beta) (Y_1,\dots,Y_k) - d^\cov (i_Z\beta)(Y_1,\dots, Y_k).
\end{aligned}
\end{equation*}
Hence we obtain the Cartan magic formula
for covariant Darboux-Lie derivative
\begin{equation}
\label{cartanmagicformula}
\begin{aligned}
\lie^\cov_{Z} \beta = i_{Z} (d^\cov \beta) + d^\cov (i_Z\beta).
\end{aligned}
\end{equation}
\section{Future work on $G$-structures}\label{sec:future}
The principal motivation for introducing the covariant Darboux–Lie derivative
was its anticipated application to the theory of $G$-structures. Recall that a
\emph{$G$-structure} on an $n$-dimensional manifold $M$ is a reduction of the
structure group of $TM$ to $G < \Gl_n(\R)$. This can be described in several equivalent ways:
\begin{enumerate}[\ \ \ $-$\ ]
    \item by choosing an open covering of $M$ such that the transition maps of $TM$ with respect to this covering take values in $G$;
    \item by specifying a principal $G$-subbundle of the frame bundle $L(\R^n, TM)$;
    \item by giving an isomorphism $TM \cong P \times_G V$, for a suitable
principal $G$-bundle $P$ and $G$-module $V$.
\end{enumerate}

The last formulation is particularly suited for analytic treatments, in contrast
with the more geometric viewpoint of describing the structure as a
subbundle\footnote{
By “analytic” we mean a formulation adapted to symbolic manipulation (compare
 \emph{analytic geometry}). This use of the word “analytic” should not be
confused with “analytic” in the sense of real or complex analysis.}.
It should be noted that two distinct isomorphisms $\beta$, $\beta' \from TM
\to P\gtimes V$ may determine the same $G$-structure. In fact, we
will formally verify that $\beta$ and $\beta'$ define the same
$G$-structure if and only if there is $s \in \Gamma(M, P\gtimes \cg)$ such
that $\beta' = s\cdot \beta$. Thus a $G$-structure on $M$ can be described as a
gauge equivalence class $\left[ \beta \right]$ of soldering forms $\beta \from
TM \to
P\gtimes V$.
We conclude this section by announcing two results that will be proved in forthcoming work.
Suppose $P$ is equipped with a fixed $G$-principal connection.
\begin{proposition}
A vector field $X \in \vf(M)$ is an infinitesimal automorphism of a
$G$-structure $\left[ \beta \right]$ on $M$ if and only if $\lie^{\cov}_X \beta
= a \cdot \beta$ for some $a \in \Gamma(M , P\gtimes \glie)$.
\end{proposition}
\begin{proposition}
A $G$-structure $\left[ \beta \right]$ is torsion-free if and only if $d^\cov
\beta = \alpha \wedge \beta$ for some $\alpha \in \Omega^1(M, P\gtimes \glie)$.
\end{proposition}

\bibliography{gstr}
\bibliographystyle{plain}

\end{document}